\documentclass{amsart}
\usepackage{amssymb,amsmath, amsthm,latexsym}
\usepackage[toc,page]{appendix} 
\usepackage{graphics}
\usepackage{psfrag}
\usepackage{amscd}
\usepackage{graphicx}
\usepackage{here}
\newcommand{\cal}[1]{\mathcal{#1}}
\theoremstyle{plain}
\newtheorem{theorem}{Theorem}[section]
 
\newtheorem{lemma}{Lemma}[section]
 
\newtheorem{proposition}[lemma]{Proposition} 
\newtheorem{corollary}[lemma]{Corollary}
\theoremstyle{definition}
 
\newtheorem{definition}[lemma]{Definition}

\usepackage{hyperref}
\usepackage{hyperref}
\hypersetup{
colorlinks=true,
linkcolor=blue,
citecolor=red
}
\usepackage{color}
\numberwithin{equation}{section}
\parskip=\bigskipamount

\let\egthree=\phi
\let\phi=\varphi
\let\varphi=\egthree
\newcommand{\setword}[2]{%
  \phantomsection
  #1\def\@currentlabel{\unexpanded{#1}}\label{#2}%
}

\begin{document}
\title[PRIMITIVE INVARIANTS FROM LAMINATIONS]
{PRIMITIVE INVARIANTS FROM LAMINATIONS}
\author{Veronica Pasquarella}
\address{Shanghai Institute for Mathematics and Interdisciplinary Sciences (SIMIS), Block A, International Innovation Plaza, No. 657 Songhu Road, Yangpu District, Shanghai, 200433, China}
\email{veronica-pasquarella@simis.cn} 
\thanks
{MSC subject classification: 57N16, 57N35, 57N65, 57T99}

\begin{abstract}
Combining geometric group theory techniques with geometric topology tools, we show how primitive cohomologies provide useful insights towards unifying the mathematical formulation of Gromov-Witten invariants. In particular, we emphasise the role played by geodesic laminations in analysing such invariants for the case of complete intersections in projective space.
\end{abstract}

\maketitle

\section{\textbf{Introduction}}

\medskip   

\medskip  

The present work is mostly concerned with understanding the geometric topology of the following group homeomorphism   

\begin{equation}   
\alpha:\ \pi_{_1}\left(\mathcal{M}_{_{g,n}},[F]\right)\ \longrightarrow\ \text{Stab}_{_{\text{MCG}_{_{g,n}}}}\left(\ell_{_{\mathcal{M}}}^{^{[F]}}\right)\ \le\ \text{MCG}_{_{g,n}} 
\label{eq:monbeg}
\end{equation}

with $\mathcal{M}_{_{g,n}}$ the moduli space of complete intersections in projective space, $[F]$ a point class in $\mathcal{M}_{_{g,n}}$, $\text{MCG}_{_{g,n}}$ its corresponding mapping class group (MCG), and $\text{Stab}_{_{\text{MCG}_{_{g,n}}}}\left(\ell_{_{\mathcal{M}}}\right)$ its stabiliser, with respect to a certain tangential structure, $\ell_{_{\mathcal{M}}}$. The main reason for pursuing this analysis is the role played by \eqref{eq:monbeg} in calculating Gromov-Witten (GW) invariants.
Our work was partly inspired by a main question addressed by \cite{RW}, namely that of describing subgroups of MCG of a hypersurface $X_{_d}\ \subset\ \mathbb{CP}^{^4}$, namely the diffeomorphisms that can be realised by monodromy  

\begin{equation}  
\alpha:\ \text{Mon}_{_d}\ \longrightarrow\ \text{MCG}_{_d}
\end{equation} 

Considering a complete intersection $\mathcal{U}_{_d}\ \subset\ \mathbb{P}H^{^0}\left(\mathbb{CP}^{^4};\mathcal{O}(d)\right)$, the monodromy group and the group of isotopy classes of orientation-preserving diffeomorphisms, are therefore respectively defined as follows  
\begin{equation}  
\text{Mon}_{_d}\ \overset{def.}{=}\ \pi_{_1}\left(\mathcal{U}_{_d}, X_{_d}\right)  \nonumber
\end{equation}  
\begin{equation}  
\text{MCG}_{_d}\ \overset{def.}{=}\ \pi_{_0}\ \text{Diff}^{^+}\left(\mathbb{CP}^{^4}; \mathcal{O}(d)\right)  \nonumber
\end{equation}  
The main point of our work is that of highlighting how primitive cohomologies arise in \eqref{eq:monbeg}, and how this impacts the geometric topology of the moduli space of such varieties, supported by the difference in GW-invariant calculation with respect to other known varieties. 
Among the cases studied by \cite{GTV} are those of \cite{MT}, namely symplectic 4-manifolds that are instrinsically related to Riemann surfaces and their respective MCG. Indeed, given a Riemann surface, $\Sigma$, and an element 
\begin{equation}   
f:\ \Sigma\ \rightarrow\ \Sigma  \nonumber
\end{equation}  
of its MCG, we can combine $(\Sigma, f)$ into 3-dimensional manifolds, the mapping torus  
\begin{equation}  
\pi:\ Y_{_f}\ \longrightarrow\ S^{^1},  \nonumber
\end{equation}  
with fiber $\Sigma$ and monodromy $f$. Then, taking the product of $Y_{_t}$ with an $S^{^1}$ leads to a symplectic 4-manifold  
\begin{equation}  
X\ = \ S^{^1}\ \times\  Y_{_f}  \nonumber
\end{equation}  
By construction, the data of this 4-manifold, $(X,\omega_{_f})$, entirely resides in the $(\Sigma, f)$ pair. The most important property of such 3-manifolds is that, in principle, could have different fibrations, with monodromy given by different elements of MCG, which, in turn, would lead to different symplectic forms on $X$. Mc Mullen and Taubes, \cite{MT}, made use of Seiberg-Witten invariants to prove that symplectic structures from different fibrations can be inequivalent, meaning there is no combination of diffeomorphisms nor smooth paths of symplectic forms interpolating between the symplectic forms resulting form different fibrations. 

The work of \cite{TY} introduced and analysed new cohomologies for compact symplectic manifolds, termed \emph{primitive cohomologies}, opening the path to the study of new cohomological invariants for a given symplectic manifold. The main advantage with respect to the ordinary de Rham cohomology is the dependence of its dimension on the symplectic structure. The pioneering work of \cite{TY} proved such dependence for the case of 6-dimensional manifolds, by showing that $PH^{^{\bullet}}_{_{\pm}}(X,\omega)$ can jump in dimension as the symplectic structure $\omega$ varies on $X$. Inspired by \cite{TY}, \cite{GTV} extend the applicability of primitive cohomologies to 4-manifolds, traced back the jump in the symplectic forms in the work of Mc Mullen and Taubes to the presence of primitive cohomologies.  

In the present work, we bring this analysis further by possibly opening a new perspective in unifying formulations of GW-invariants, combining tools from hyperbolic geometry, geometric topology, algebraic topology and geometric group theory. Among one of our main points is highlighting the relation between lack of residual finiteness of MCG in \cite{RW}, and primitive cohomologies (The latter naturally arising in the context of immersions of complete intersections in projective space.). Specifically, \cite{RW} shows that MCG is not always residually finite. 

In practice, we analyse GW-invariants for complete intersections in projective space. The latter were first studied by \cite{ABPZ}, which highlights the main difference with respect to previously known cases. In particular, the main ingredient that was added in the work of \cite{ABPZ} was a graph, $\Gamma$, featuring as additional data dressing the moduli space, $\mathcal{M}_{_{g,n}}$, of the variety in question. We refer to the main sections of the present work for further details of this assignment. However, as also pointed out by the authors, more general invariants could in principle be achieved if one were to allow a dependence of the primitive cohomologies on the symplectic structures. From the arguments outlined above, we argue that analysing such a case constitutes a major step towards achieving a universal understanding of the formulation of GW-invariants, as well as its impact towards extending Homological Mirror Symmetry, which we will briefly comment upon in this work, while referring for further details in a follow-up paper by the same author.  

The main results of the present work can be summarised by means of the following theorems.

\begin{theorem}\label{Theorem:1.1}  For complete intersections in projective space, the dimension of the Thurston spine is bounded below by the virtual cohomological dimension of the image in MCG of the fundamental germ of the lamination, $\mathcal{L}$, replacing the primitive insertions
\begin{equation}   
\text{VCohdim}\left( 
\alpha\left([|\pi|]_{_1}(\mathcal{L}), [F]\right)\right)\ \le\ \text{dim}\ \mathcal{P}_{_{g,n}}\bigg|_{_{[F]}} 
\label{eq:thsp1intr}
\end{equation}
\end{theorem}

The equality in \eqref{eq:thsp1intr} is achieved in the limit that the primitive cohomological insertions are independent with respect to the symplectic form, which would thereby correspond to $\text{VCohdim}\ \text{MCG}_{_{g,n}}$. 

The proof of Theorem \ref{Theorem:1.1} requires some preliminary intermediate steps, in turn leading to formulating the following statements.

\begin{theorem}  The lamination replacing primitive cohomology insertions is not a critical point of the systole function.  
\end{theorem}

\begin{theorem}  

\begin{equation}  
MCG_{_{g,n}}^{^{\Gamma_{_i}}}\ <\ MCG_{_{g,n}}^{^p}  
\end{equation}  
with the left-hand-side featuring constrained pair of pants decompositions.
\end{theorem}

\begin{theorem}   

The following statements are equivalent:

\begin{itemize}  

\item MCG is non-separable if the simple cohomologies depend on the nodes (namely if the symplectic form depends on the primitive cohomology insertion).

\item  Identifying MCG$_{_{g,n}}$ for $\mathcal{M}_{_{g,n}}^{^{\Gamma}}$ with primitive cohomology insertions is an example of an unsolvable word problem, featuring a nonrecursive distorsion for subgroups associated to each node of the graph, $\Gamma$, dressing the moduli space.

\end{itemize}    

\end{theorem}

\begin{theorem}  Determining MCG$_{_{g,n}}$ for $\mathcal{M}_{_{g,n}}^{^{\Gamma}}$ with primitive cohomology insertions is an example of an unsolvable word problem, featuring a nonrecursive distorsion for subgroups associated to each node of the graph, $\Gamma$.
\end{theorem}

\begin{theorem}  The jump in the topological entropy of the pseudo-Anosov map generating MCG$_{_{g,n}}$ is related to a distance in MCG$_{_{g,n}}$, quantifying a subgroup's concavity.
\end{theorem} 

\begin{theorem}  

Complete intersections in projective space always feature a pseudo-Anosov map whose topological entropy is bounded from below by the logarithm of its spectral radius.

\end{theorem}

\begin{theorem}  
If MCG$_{_{g,n}}^{^p}$ is non-separable, $h^{^{\bullet}}$ is not pseudo-Anosov.
\end{theorem}  

\begin{theorem}   For complete intersections in projective space,   
\begin{equation}   
\alpha\left([|\pi|]_{_1}(\mathcal{L}), [F]\right)\ \le\ \text{Im}(\alpha)\ \le\ \text{Stab}_{_{\text{MCG}_{_{g,n}}}}\left(\ell_{_{{\mathcal{M}}}}^{^{[F]}}\right)
\end{equation}     
with the equality being achieved when $\mathcal{L}$ corresponds to a pants decomposition of $\mathcal{M}_{_{g,n}}$. The latter essentially corresponds to a saturation limit.
\end{theorem}

The present work is structured as follows: Section \ref{sec:1} introduces deformation retraction from the point of view of homotopy theory, as well as in the original formulation put forward by Thurston for genus-$g$ curves with and without punctures. We briefly overview the notions of the systole function, $f_{_{\text{sys}}}$, as one of three metric invariants of interest to us. We then turn to explaining how primitive cohomologies arise in the algebro topologic analysis of \cite{RW}, and how this can be used to address the calculation of GW-invariants for complete intersections in projective space, to which later Sections will be devoted. 

Section \ref{sec:2} briefly overviews GW-invariants with primitive cohomology insertions in two main cases. At this stage, we introduce two other metric invariants, namely the injectivity radius, and the generalised systole function. These tools will be useful for proofs of statements in later sections. The remainder of the Section comments on the impact of primitive cohomology insertions on topological recursion of GW-invariants for complete intersections in projective space.

Section \ref{sec:5} Introduces the main ingredient of our analysis, namely laminations, in the context of geometric group theory. The crucial point of the present work is highlighting how such geometric topology tool can be used to gain insight in the algebro-geometric analysis of complete intersections in projective spaces.  The main group of interest to us is the MCG of complete intersections in projective space, properties of its subgroups, and relations among them. In particular, we focus on the relation in between the MCG of the moduli space of the varieties in question and how primitive cohomology insertions affect the geometric group theory properties of  the resulting MCG.

Section \ref{sec:6} ultimately combines tools from previous parts of the present work and relate this to the construction of the obstruction bundle, naturally arising in the calculation of invariants of algebraic varieties with singular moduli space. In so doing, we show how it relates with the work of \cite{KKS}, while providing promising new tools that could potentially aid towards generalising the mathematical formulation of GW-invariants for arbitrary varieties. In particular, we highlight the importance of the lack of residual finiteness, and the role of the fundamental germ in determining a lower bound on the dimension of the Thurston spine.  Making use of metric invariants defined in Section \ref{sec:2}, we also explain how a generalised entropy theorem on such varieties can be extracted from previous literature, in a way that could potentially aid towards furthering homological mirror symmetry understandings.    

\tableofcontents

\section{\textbf{Equivariant deformation retraction}}  \label{sec:1}

\subsection{The fundamental group and deformation retraction}

\begin{definition}   Let $X$ be a space; let $x_{_o}$ be a point of $X$. A path in $X$ that begins and ends at $x_{_o}$ is a loop based at $x_{_o}$. The set of path homotopy classes of loops based at $x_{_o}$, with operation. $*$ is the fundamental group of $X$ relative to the point $x_{_o}$, $\pi_{_1}\left(X,x_{_o}\right)$. $*$ denotes the product operation on paths, inducing a well-defined operation on equivalence classes as follows  
\begin{equation}  
[f]*[g]=[f*g] \nonumber 
\end{equation} 
\end{definition} 

\begin{definition} Let $A$ be a subspace of $X$. $A$ is said to be a deformation retract of $X$ if there is a continuous map   
\begin{equation}  
H:\ X\ \times\ I\ \longrightarrow\ X  
\nonumber
\end{equation}     
such that     
\begin{equation}    
H(x,0)\ =\ x\ \ \ \ , \ \ \ \ H(x,1)\ \in\ A\ \ \ \ \ \ \ \forall\ x\ \in\ X
\nonumber
\end{equation}  
and 
\begin{equation}  
H(a,t)\ =\ a\ \ \ \ \forall\ a\ \in\ A  \nonumber
\end{equation} 
The homotopy $H$ is called deformation retraction of $X$ onto $A$. The map 
\begin{equation}  
r(x)\ =\ H(x,1)  \nonumber
\end{equation}  
is a retraction of $X$ onto $A$, and $H$ is a homotopy between the identity map of $X$ and the map $j\circ\ r$, where  
\begin{equation}  
j:\ A\ \longrightarrow\ X  \nonumber
\end{equation}   
is the inclusion map.      
\end{definition}

\begin{theorem}  Let $A$ be a deformation retract of $X$, and $x_{_o}\in A$. Then, the inclusion map  
\begin{equation}  
j:\ \left(A,x_{_o}\right)\ \longrightarrow\ \left(X,x_{_o}\right)  \nonumber
\end{equation}  
induces an isomorphism of fundamental groups. 
\end{theorem}

\begin{theorem} [Van Kampen]   Let $X=U\cup V$, $U,V$ open in $X$; assume that $U,V$ and $U\cap V$ are path-connected. Let $x_{_o}\in U\cap V$. Let $H$ be a group, and  
\begin{equation}  
\phi_{_1}:\ \pi_{_1}(U,x_{_o})\ \longrightarrow\ H  
\ \ \ \ \ \ \ , \ \ \ \ \ \ \     
\phi_{_2}:\ \pi_{_1}(V,x_{_o})\ \longrightarrow\ H \nonumber
\end{equation}  
be a homeomorphism. Let $i_{_1}, i_{_2}, j_{_1}, j_{_2}$ be the homeomorphisms indicated in the following diagram, each induced by inclusion.

\begin{equation}   
\begin{aligned}    
&\ \ \ \ \ \ \ \ \ \ \ \ \ \ \ \ \ \ \ \ \ \ \ \ \ \ \ \ \ \ \ \ \pi_{_1}(U,x_{_o})\\  
&\ \ \ \ \ \ \ \ \ \ \ \ \ \ \    i_{_1}\ \ \nearrow\ \ \ \ \ \ \ \  \downarrow\ j_{_1} \ \ \    \ \ \ \ \ \ \ \searrow\ \alpha_{_1}\\
&\pi_{_1}(U\cap V, x_{_o})\ \longrightarrow\ \pi_{_1}(X, x_{_o})\ \color{red}\longrightarrow\ \color{black}H\\    
&\ \ \ \ \ \ \ \ \ \ \ \ \ \ \  i_{_2}\ \ \searrow\ \ \ \ \ \ \ \  \uparrow\ j_{_2} \ \ \  \ \ \ \ \ \   \ \nearrow\ \alpha_{_2}\\
&\ \ \ \ \ \ \ \ \ \ \ \ \ \ \ \ \ \ \ \ \ \ \ \ \ \ \ \ \ \ \ \ \pi_{_1}(V,x_{_o})\\
\end{aligned}  
\end{equation}     
If
\begin{equation}  \alpha_{_1}\ \circ\ i_{_1}\ =\ \alpha_{_2}\ \circ\ i_{_2}  \nonumber
\end{equation}  
there is a unique homeomorphism  
\begin{equation}  
\alpha:\ \pi_{_1}(X, x_{_o})\ \longrightarrow\ H  \nonumber
\end{equation}     
such that   
\begin{equation}  
\alpha\ \circ\ j_{_1}\ =\ \alpha_{_1}\ \ \ \ \ \ \ , \ \ \ \ \ \ \ \ \alpha\ \circ\ j_{_2}\ =\ \alpha_{_2}  \nonumber
\end{equation}   
\end{theorem}

\subsection{Thurston's deformation retraction}    

\begin{definition}  Teichm$\ddot{\text{u}}$ller space, $\mathcal{T}_{_{g,n}}$, is the space of all conformal equivalence classes of marked Riemannian surfaces of genus $g$ with $n$ punctures.  The Teichm$\ddot{\text{u}}$ller space of a surface is a contractible space on which the MCG acts properly-discontinuously. The minimal space on which MCG acts properly-discontinuously is achieved by deformation retraction. 

\end{definition}   

\begin{definition} A systole in a closed hyperbolic surface is a closed geodesic of minimal length.
\end{definition}   

\begin{definition}  The systole function 
\begin{equation}  
f_{_{\text{sys}}}:\ \mathcal{T}_{_g}\ \longrightarrow\ \mathbb{R}_{_+}    
\end{equation}  
is the function that maps $x\in\mathcal{T}_{_g}$ to the length of the systoles at $x$. $f_{_{\text{sys}}}$ is a topological Morse function, featuring a finite number of critical points. Specifically, the number of such critical points is twice the number of pair of pants decompositions.

\end{definition}

\begin{definition} A set of closed geodesics fills if each component of the complement of their union is contractible, namely if it cuts the surface into polygons.
\end{definition}  
Thurston proved that the moduli space of closed hyperbolic surfaces of genus $g$ deformation retracts into the set of closed hyperbolic surfaces with filling systoles, \cite{Th}.

\begin{definition}The mapping class group, MCG$_{_{g,n}}$, acts properly-discontinuously on $\mathcal{T}_{_{g,n}}$, and stabilisers of points are all finite. Then,  
\begin{equation}   
\mathcal{M}_{_{g,n}}\ \overset{def.}{=}\ \mathcal{T}_{_{g,n}}/\text{MCG}_{_{g,n}}  
\end{equation}  
is the moduli space of curves, whose points are conformal equivalence classes of unmarked Riemann surfaces of genus $g$ with $n$ punctures. 

\end{definition}  

\begin{definition}  A CW-complex is a combinatorial object, playing a crucial role in Algebraic Topology. Indeed, most spaces of interest to Algebraic Topologists are homotopy equivalent to CW-complexes. C- stands for closure finiteness, and W for weak topology.   

Among the main features of CW-complexes is the fact that they are locally contractible topological spaces.

\end{definition}  

\begin{definition}  Deformation retraction is a shrinking of $\mathcal{T}_{_g}$ to a smaller space on which MCG still acts properly-discontinuously.

\end{definition}

\begin{definition}  The image of deformation retraction is referred to as the Thurston spine, $\mathcal{P}_{_g}$. It is a CW-complex contained in $\mathcal{P}_{_g}$, as the collection of points representing hyperbolic surfaces cut into polygons by  the set of systoles. Therefore, $\mathcal{P}_{_g}$ contains the critical points of $f_{_{\text{sys}}}$.

\end{definition}

\begin{proposition}  Determining the MCG corresponds to identifying a preferred filling lamination, the Thurston spine. 
\end{proposition}
\begin{proof}     This follows from the definition of the Thurston spine, together with the fact that the latter can be built from two filling multi-curves. One such multicurve can be suitably taken to be a lamination providing a pair of pants decomposition of $\mathcal{M}_{_{g,n}}$.
\end{proof}

\begin{proposition}  Deformation retraction is an extremal limit (a global limit) of the obstruction bundle associated to singular moduli spaces.
\end{proposition}   
\begin{proof}       
To find the obstruction theory associated to the theory, we need to identify a critical point of $f_{_{\text{sys}}}$, defining the corresponding obstruction theory.  
\begin{equation}  
\text{dim}\ \mathcal{P}_{_g}\ =\ \text{VCohdim}\left(\text{MCG}_{_g}\right)  
\label{eq:thusrt}
\end{equation}
\end{proof}    
\begin{proposition} [Thurston]  Let $C$ be a collection of curves on a surface that do not fill. Then, at any point of $\mathcal{T}_{_g}$, there are tangent vectors that simultaneously increase the lengths of all geodesics representing curves in $C$.
\end{proposition} 
From this Proposition, it follows that all critical points of $f_{_{\text{sys}}}$ are contained in $\mathcal{P}_{_g}$. 
A similar analysis can be carried out in presence of punctures, \cite{H}.
\subsection{Stable homotopy theory}    
\subsubsection{Tangential structures}
\begin{definition}  Given a smooth manifold, $X$, a tangential structure is a lift of the classifying map  
\begin{equation}  
X\ \longrightarrow\ BGL(n)  \nonumber
\end{equation}    
of its tangent bundle through any prescribed map  
\begin{equation}  
f:\ B\ \longrightarrow\ BGL(n)  \nonumber
\end{equation}   
into the classifying space of the general linear group, up to homotopy:  
\begin{equation}    
\begin{aligned}    
&\ \ B\\
\text{tang. str.}\nearrow  \ \ &\ \ \downarrow\ f\\
X\ \longrightarrow\ &BGL(n)  \\    
\end{aligned}  \nonumber
\label{eq:commdiagr}   
\end{equation} 
\end{definition}  

In the context of the homotopy type of topological spaces, there is a weak topology equivalence   
\begin{equation}    
BGL(n)\ \simeq\ BO(n)   \nonumber
\end{equation}    
to the classifying space of the orthogonal group, the following diagram is equivalent to \eqref{eq:commdiagr}, 
\begin{equation}    
\begin{aligned}    
&\ \ B_{_d}\\
\nearrow  \ \ &\ \ \downarrow\ \theta\\
M\ \xrightarrow{\tau_{_M}}\ &BO(d)  \\    
\end{aligned}\nonumber
\end{equation}  

\begin{definition}  A $(B,f)$-structure is a pointed CW-complex 

\begin{equation}  
B_{_n}\ \in\ \text{Top}_{_{CW}}^{^*}/  \nonumber
\end{equation}  
equipped with a pointed Serre fibration to the classifying space $BO(n)$. 

\begin{equation}  
f_{_n}:\ B_{_n}\ \longrightarrow\ BO(n)  \nonumber
\end{equation}  
\end{definition}

Examples of $(B,f)$-structures include the following:  

\begin{enumerate}

\item $B_{_n}\ =\ BO(n)$, $f_{_n}\ =\ \text{id}$ is the orthogonal structure.     

\item $B_{_n}\ =\ EO(n)$, $f_{_n}$ the universal principal bundle projection (namely a framing structure).  

\item $B_{_n}\ =\ BSO(n)\ =\ EO(n)/SO(n)$, $f_{_n}$ the canonical projection provides an orientation structure.  

\item $B_{_n}\ =\ B\text{Spin}(n)\ =\ EO(n)/\text{Spin}(n)$, $f_{_n}$  is a spin structure. 

\end{enumerate}    

\begin{definition} Given a smooth manifold, $X$, of dimension $n$, and given a $(B,f)$-structure, then a $(B,f)$-structure on the manifold is an equivalence class of the following structure:  
\begin{itemize}  
\item An embedding  
\begin{equation}  
\iota_{_X}:\ X\ \hookrightarrow\ \mathbb{R}^{^k}  \nonumber
\end{equation}  
for some $k\in\mathbb{N}$. 
\item A homotopy class of a lift $\hat g$ of the classifying map $g$ of the tangent bundle  
\begin{equation}    
\begin{aligned}    
&\ \ B_{_n}\\
\hat g\nearrow  \ \ &\ \ \downarrow\ \theta\\
X\ \xrightarrow{g}\ &BO(n)  \\    
\end{aligned}   \nonumber
\end{equation}    
\end{itemize}     
\end{definition}

\begin{definition} In homotopy theory, an H-space consists of: a type $A$, a basepoint, $e$, a binary operation  
\begin{equation}  
\mu:\ A\ \longrightarrow\ A\ \longrightarrow\ A   ,\nonumber
\end{equation}  
a left- and right-unitor  
\begin{equation}  
\lambda:\ \prod_{_{(a:A)}}\ \mu(e,a)\ =\ a  \ \ \ \ \ \ \ \ \ \ \ \   
\rho:\ \prod_{_{(a:A)}}\ \mu(a,e)\ =\ a  \nonumber
\end{equation}  
\end{definition}  

In the notation of \cite{RW}, this translates as follows.

\begin{definition}  A stable tangential structure is a Serre fibration  
\begin{equation}   
\overline\theta:\ B\ \longrightarrow\ BO  \nonumber
\end{equation}  
\end{definition}  
A $\overline\theta$-structure on a manifold $M$ is a choice of a lift    
\begin{equation}  
\ell:\ M\ \longrightarrow\ B  \nonumber
\end{equation}  
of the map  
\begin{equation}  
\tau_{_M}^{^S}:\ BSO\ \longrightarrow\ BO  \nonumber
\end{equation}
However, the map $\tau_{_M}$ is not unique, therefore, in order to study the action of Diff$(M)$ on $\overline\theta$-structures, it is best to consider the space of $\theta$ -structures on $M$, defined as the space of bundle maps  
\begin{equation}  
\Theta(M)\ \overset{def.}{=}\ \text{Bun}\left(TM, \theta^{^*}\gamma_{_d}\right)  \nonumber
\end{equation}
that is homotopically equivalent to the space of lifts of the diagram 
\begin{equation}    
\begin{aligned}    
&\ \ B_{_d}\\
\nearrow  \ \ &\ \ \downarrow\ \theta\\
M\ \xrightarrow{\tau_{_M}}\ &BO(d)  \\    
\end{aligned}   \nonumber
\end{equation}
Given a complete intersection in projective space
\begin{equation}  
X_{_d}\ \subset\ \mathbb{CP}^{^4}.  \nonumber
\end{equation}
Treating $X_{_d}$ as a degree-$d$ hypersurface, its tangent bundle satisfies the following identity  
\begin{equation}   
TX\ \oplus\ \mathcal{O}(d)\bigg|_{_X}\ =\ T\mathbb{CP}^{^4}\bigg|_{_X}  
\label{eq:tang}    
\end{equation}  
Importantly, $X_{_d}$ comes equipped with two natural tangential structures, that behave differently under the action of MCG.  

\begin{itemize}

\item   A $\theta^{^{\mathbb{C}}}$-structure. 

\item  A $\theta^{^{\text{hyp}}}$-structure, $\ell_{_X}^{^{\text{hyp}}}$, induced by the maps  

\begin{equation}  
\overline\theta^{^{\text{hyp}}}:\ \mathbb{CP}^{^{\infty}}\ \xrightarrow{}\\ BU\ \xrightarrow{\overline\theta^{^{\mathbb{C}}}}\ BO  
\end{equation}
\end{itemize}   

While both are preserved under the action of Mon$_{_d}$, the latter is not preserved by MCG$_{_d}$, the reason being that 

\begin{equation}   
\overline\theta^{^{\text{hyp}}}:\ \mathbb{CP}^{^{\infty}}\ \longrightarrow\ BO  
\label{eq:stable}    
\end{equation}   
is no longer a map of $H$-spaces.   

The action of MCG$_{_d}$ on the set $\theta^{^{\text{hyp}}}\left(X_{_d}\right)$ preserves the subset $\theta^{^{\text{hyp}}}\left(X_{_d};x,+\right)$ of those $\theta^{^{\text{hyp}}}$-structures which map to $x\in H^{^2}\left(X_{_d};\mathbb{Z}\right)$ under$\hat X$ and that induce the standard orientation of $X_{_d}$. This action descends to an action of Aut$\left(\pi_{_3}\left(X_{_d}\right),\lambda\right)$ on the set  

\begin{equation}  
\theta^{^{\text{hyp}}}\left(X_{_d};\mathbb{Z}\right)\bigg/\text{SMCG}_{_d}  \nonumber
\end{equation}    
This set is a torsor for 
\begin{equation}  
H^{^3}\left(X_{_d};\mathbb{Z}/2\right)\ =\ H^{^3}\left(X_{_d};\pi_{_3}(O)\right)\bigg/\text{Im}\left(\delta^{^{\text{hyp}}}\bigg|_{_{\text{SMCG}_{_d}}}\right)  \nonumber
\end{equation}   
To a $\theta^{^{\text{hyp}}}$ structure, $\ell$, there is an associated quadratic form    
\begin{equation}     
\mu_{_{\ell}}:\ \pi_{_3}\left(X_{_d}\right)\ \longrightarrow\ \mathbb{Z}/2  
\end{equation}  
on  
\begin{equation}  
\pi_{_3}\left(X_{_d}\right)\ \xleftarrow{\sim}\ \pi_{_4}\left(\mathbb{CP}^{^{\infty}}, X_{_d}\right)\ \xrightarrow{\sim}\ H_{_4}\left(\mathbb{CP}^{^{\infty}},X_{_d};\mathbb{Z}\right)  
\end{equation}  
obtained by considering this as a surgery kernel for the normal map  
\begin{equation}  
\ell:\ X_{_d}\ \longrightarrow\ \mathbb{CP}^{^{\infty}}  
\end{equation}   
covered by the corresponding bundle map. This $\mu_{_{\ell}}$ is a quadratic refinement of the intersection form $\lambda$, in the sense that   
\begin{equation}  
\mu_{_{\ell}}(a+b)\ =\ \mu_{_{\ell}}(a)+\mu_{_{\ell}}(b)\ +\ \lambda(a,b)\ \text{mod} 2  \nonumber
\end{equation}  

$\mu_{_{\ell_{_{X_{_d}}}^{^{\text{hyp}}}}}$ is Mon$_{_d}$-invariant. The image of Stab$_{_{\text{MCG}_{_d}}}\left(\ell_{_{X_{_d}}}^{^{\text{hyp}}}\right)$ in Aut$\left(\pi_{_3}\left(X_{_d}\right),\lambda\right)$ is the stabiliser Aut$\left(\pi_{_3}\left(X_{_d}\right),\lambda,\mu\right)$ of the quadratic form $\mu\equiv\mu_{_{\ell_{_{X_{_d}}}^{^{\text{hyp}}}}}$ of the intersection form $\lambda$. 
The stable tangential structure \eqref{eq:stable} is not the most sophisticated structure we can endow a hypersurface with. Indeed, \eqref{eq:tang} shows that $X_{_d}$ can be endowed with the unstable tangential structure given by the homotopy equaliser  
\begin{equation}  
B\ \longrightarrow\ BU(3)\ \times\ \mathbb{CP}^{^4}\ \rightrightarrows\ BU(4)  \nonumber
\end{equation}  
of the maps classifying   
\begin{equation}  
\gamma_{_3}^{^{\mathbb{C}}}\ \oplus\ \mathcal{O}(d)  \nonumber
\end{equation}  
and $T\mathbb{CP}^{^4}$, respectively, made into a tangential structure via the natural maps  

\begin{equation}  
B\ \longrightarrow\ BU(3)\ \times\ \mathbb{CP}^{^4}\ \longrightarrow\ BU(3)\ \longrightarrow\ BO(6)  \nonumber
\end{equation}

\begin{equation}  
\delta^{^{\text{hyp}}}:\ \text{MCG}_{_d}\ \longrightarrow\ H^{^3}\left(X_{_d};\ \pi_{_3}(O)\right)\ \subset\ KO^{^{-1}}\left(X_{_d}\right)  \nonumber  
\end{equation}  

\begin{equation}  
\text{Stab}_{_{\text{MCG}_{_d}}}\left(\ell_{_{X_{_d}}}^{^{\text{hyp}}}\right)\ \overset{def.}{=}\ \text{Ker}\left(\delta^{^{\text{hyp}}}:\ \text{MCG}_{_d}\ \longrightarrow\ H^{^3}\left(X_{_d};\ \pi_{_3}(O)\right)\ \subset\ KO^{^{-1}}\left(X_{_d}\right)\right)    \nonumber
\end{equation}    
By further denoting with
\begin{equation}  
\text{SMCG}_{_d} \ \le\ \text{MCG}_{_d}  
\end{equation}  
the subgroup of diffeomorphisms that are trivially acting on $\pi_{_3}(X_{_d})$, 
\begin{equation}  
\text{Stab}_{_{\text{MCG}_{_d}}}\ \cap \ \text{SMCG}_{_d}\ = \ K_{_d}  
\end{equation} 
leading to a half-exact sequence  
\begin{equation}  
1\ \longrightarrow\ K_{_d}\ \longrightarrow \text{Stab}_{_{\text{MCG}_{_d}}}\left(\ell_{_{X_{_d}}}^{^{\text{hyp}}}\right)\ \longrightarrow\ \text{Aut}\left(\pi_{_3}\left(X_{_d}\right),\lambda\right)
\end{equation}
which can be turned into an exact sequence upon restricting to the image of $\text{Stab}_{_{\text{MCG}_{_d}}}\left(\ell_{_{X_{_d}}}^{^{\text{hyp}}}\right)$ in $\text{Aut}\left(\pi_{_3}\left(X_{_d}\right),\lambda\right)$ that stabilises the quadratic form  
\begin{equation}  
\mu\ =\ \mu_{_{\ell_{_{X_{_d}}}^{^{\text{hyp}}}}}  \nonumber
\end{equation}  
\begin{equation}  
1\ \longrightarrow\ K_{_d}\ \longrightarrow \text{Stab}_{_{\text{MCG}_{_d}}}\left(\ell_{_{X_{_d}}}^{^{\text{hyp}}}\right)\ \longrightarrow\ \text{Aut}\left(\pi_{_3}\left(X_{_d}\right),\lambda, \mu   \right)\ \longrightarrow\ 1
\end{equation}   
$\text{Stab}_{_{\text{MCG}_{_d}}}\left(\ell_{_{X_{_d}}}^{^{\text{hyp}}}\right)$ is the largest subgroup of MCG$_{_d}$ commuting with it. 
\begin{figure}[ht!]    
\begin{center}
\includegraphics[scale=0.75]{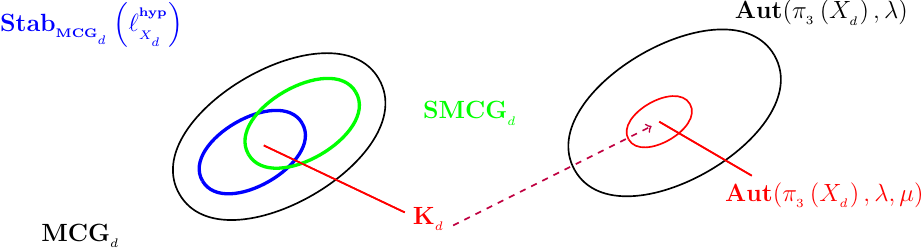}   
\caption{\small   This picture helps visualise some essential features of the work of \cite{RW}, that we will often be using. What is represented in this picture holds for all complete intersections in projective space. Notation is as explained in our treatment.}
\label{fig:charts} 
\end{center}
\end{figure} 
Upon extracting a similar short-exact sequence for \eqref{eq:monbeg}, the action of the stabiliser can therefore be equivalently considered as the deformation retraction on Teichm$\ddot{\text{u}}$ller space   
\begin{equation}  
\text{Stab}_{_{\text{MCG}_{_{g,n}}}}\left(\ell_{_{\mathcal{M}}}^{^{[F]}}\right)\ \overset{def.}{=}\ \Phi:\ \mathcal{T}_{_{g,n}}\ \longrightarrow\ \mathcal{P}_{_{g,n}}  
\end{equation}  
with corresponding Thurston spine       
\begin{equation}   
\mathcal{P}_{_{g,n}}\ =\ \Phi\left(\mathcal{T}_{_{g,n}}\right)\ =\ \text{Stab}_{_{\text{MCG}_{_{g,n}}}}\left(\ell_{_{\mathcal{M}}}^{^{[F]}}\right)\left(\mathcal{T}_{_{g,n}}\right),  \nonumber
\end{equation}  
This is one of the crucial observations supporting the Proof of Theorem \ref{Theorem:1.1}, which will be completed in Section \ref{sec:6}. As a particular limit of Theorem \ref{Theorem:1.1}, we correctly recover the following identity
\begin{equation}  
\text{VCohdim}\left(\text{MCG}_{_{g,n}}\right)\ =\ \text{VCohdim}\left(\text{Stab}_{_{\text{MCG}_{_{g,n}}}}\left(\ell_{_{\mathcal{M}}}^{^{[F]}}\right)\right)\ =\ \text{dim}\mathcal{P}_{_{g,n}}  
\end{equation}  
thereby providing an alternative proof to Thurston's conjecture, \cite{Th}.

\section{\textbf{A hyperbolic geometric perspective}}  \label{sec:2}    
As already anticipated in the Introduction, the calculation of GW-invariants for complete intersections in projective space has already been carried out by from a purely algebro-geometric perspective, \cite{ABPZ}. Our work heavily relies on some of their results, while at the same time providing an alternative perspective (from geometric topology), and mostly addressing with new tools a case that was not previously analysed in \cite{ABPZ}. Let us first outline the two main cases that one could be interested in studying for the varieties in question. 

The the major difference with respect to GW-calculations for other varieties is the presence of primitive cohomologies, that are intrinsic to  the complete intersection, and are not inherited from the ambient space. Following from the pioneering work of Jun Li, the authors of \cite{ABPZ} showed that a degeneracy formula can indeed be achieved upon replacing primitive cohomologies with a graph, $\Gamma$, whose vertices come equipped with a genus function and a $\beta$-function associated to second homologies. We refer to the original work \cite{ABPZ} where such properties are outlined in full detail. Composition of vertices ensures that upon bringing all vertices together, one would be able to recover the original GW-invariant. 
Essentially, the number of edges and vertices of the graph $\Gamma$ corresponds to that of the primitive insertions. GW-invariants are said to be \emph{simple} if all cohomological insertions, $\alpha_{_i}$, are simple. 

The calculation from algebro geometric techniques leads to two possible cases:

\setword{\textbf{Case 1}}{Word:case1}   The first case is that in which primitive cohomologies are completely determined by monodromies, \cite{ABPZ}, and thereby corresponds to a saturation of the map \eqref{eq:monbeg}. The corresponding GW-invariant, written as an integral over the virtual fundamental class of the moduli space, thereby reads as follows 
\begin{equation}  
\int_{_{[\overline{\mathcal{M}}_{_{\Gamma}}(X)]^{^{\text{vir}}}}}\prod_{_{i=1}}^{^n}\psi_{_i}^{^{k_{_i}}}\text{ev}_{_i}^{^*}\left(\alpha_{_i}\right)\prod_{_h}\psi_{_h}^{^{k_{_h}}}  \nonumber
\end{equation}   
where $\alpha_{_i}$ denote the cohomologies dressing the $n$ punctures on $\mathcal{M}_{_{g,n}}$, whereas the $\psi_{_i}$s denote the Chern characters. The major feature is the presence of additional Chern characters, where $h$ labels the internal edges of the graph.
If $\alpha_{_i}$ is independent w.r.t. $h$, it means that any decomposition and re-arrangement of the punctures between the vertices of $\Gamma$ is acceptable.      

\textbf{\setword{Case 2}{Word:case2}}     The second case adds an explicit dependence of the simple cohomologies to the edges of the graph $\Gamma$, 
\begin{equation}  
\int_{_{[\overline{\mathcal{M}}_{_{\Gamma}}(X)]^{^{\text{vir}}}}}\prod_{_{i=1}}^{^n}\psi_{_i}^{^{k_{_i}}}\prod_{_h}\psi_{_h}^{^{k_{_h}}}  \text{ev}_{_i}^{^*}\left(\alpha_{_i}(h)\right) \nonumber
\end{equation}   

As a consequence of this, the primitive cohomologies and monodromies needs to be generalised, to which part of Section \ref{sec:6} will be devoted. As we shall see, Case 1 will be recovered as a limiting case of this. Section \ref{sec:6} will be making use of some key observations made in Section \ref{sec:1}, in particular with regard to deformation retraction. For the time being, the remainder of this Section will be overviewing some additional tools we shall be needing for later purposes, especially the notion of the injectivity radius, and the generalised systole function.  
 \begin{figure}[ht!]    
 \begin{center}     
 \includegraphics[scale=0.46]{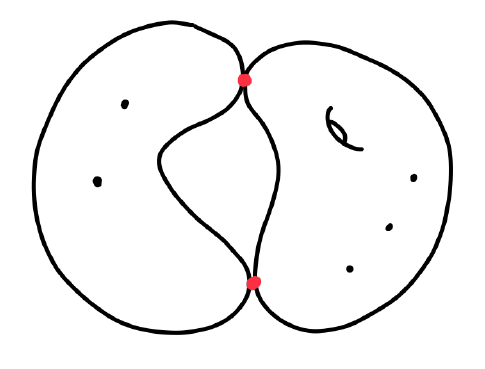}  \ \ \ \ \ \ \ \ \ \ \ \ \ \ \ \ \ 
 \includegraphics[scale=0.4]{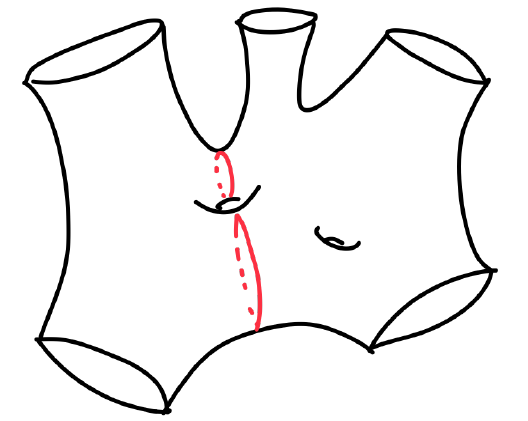}  
 \label{fig:lamination}    
 \caption{\small In red, we denote the location of the primitive cohomology insertions. The latter can be replaced by simple cohomology insertions at the price of adding an explicit graph dependence on the moduli space of the variety in question. The picture featuring on the left-hand-side is analogous to one presented in \cite{ABPZ}. On the right-hand-side, instead, we present a major model our work will keep referring to, but we emphasise that our statements apply to more general setups, with higher genus, punctures, and number of primitive cohomology insertions.}
 \end{center}  
 \end{figure}  

We conclude the Section by commenting on some expectations regarding topological recursion. 
 \subsection{Generalised systole function and injectivity radius}  

\begin{definition} A Hadamard space is a simply-connected Riemannian manifold, $(X,d)$ with non-positive sectional curvatures.    
\end{definition}    
In geometry, Hadamard spaces are non-linear generalisations of Hilbert spaces, oftentimes, equivalently referred to as CAT (0) spaces.  
Given  
\begin{equation}  
\Gamma\ <\ \text{Isom}(X)    \nonumber
\end{equation}  
finitely generated, discrete, and torsion-free, the associated Riemannian manifold reads
\begin{equation}  
M\ =\ X/\Gamma.  \nonumber
\end{equation}  
Given $p\in X$, the critical exponent of $\Gamma$ (or $M$) is  
\begin{equation}
\delta(M)\ \equiv\ \delta(\Gamma)\ \overset{def.}{=}\ \text{inf}\left\{\ s:\sum_{_{\gamma\in\Gamma}}e^{^{-sd(p,\gamma p)}}<\infty\right\}  
\label{eq:delta}  
\end{equation}  
Importantly, \eqref{eq:delta} can be used to control a wide range of geometric and topological properties of $M$, as well as providing upper bounds on the (co)homological dimension.    
\begin{definition} The injectivity radius of a complete Riemannian manifold is the largest $r$ such that every geodesic of length less than or equal to $r$is a shortest curve between its endpoints.    

\end{definition}

Let $\mathcal{M}$ be a complete Riemannian manifold with Riemannian metric, $g$, and induced distance, $d$. The exponential map at $x$ is   
\begin{equation}    
\begin{aligned}
\text{Exp}_{_x}:\ &T_{_x}M\ \longrightarrow\ \mathcal{M}    \\    
&\ \ \ \xi\ \ \ \ \  \mapsto\ \ \gamma_{_{\xi}}(1)
\end{aligned}
\end{equation}
where $T_{_x}M$ denotes the tangent space to $\mathcal{M}$ at $x$, whereas  
\begin{equation}  
\gamma_{_{\xi}}:\ [0,\infty)\ \longrightarrow\ \mathcal{M}  \nonumber
\end{equation}   
is the geodesic of $\mathcal{M}$ satisfying  
\begin{equation}  
\gamma(0)\ =\ x\ \ \ \ , \ \ \ \ \gamma^{^{\prime}}(0)\ =\ \xi  \nonumber
\end{equation}   
Denoting by    
\begin{equation}  
U_{_x}\mathcal{M}\ \equiv\ \{\xi\in T_{_x}\mathcal{M}\ |\ g_{_x}(\xi,\xi)=1\}  \nonumber  
\end{equation}
the unit tangent space at $x\in\mathcal{M}$, given $\xi\in U_{_x}\mathcal{M}$, the cut time of $\gamma_{_{\xi}}$ is  
\begin{equation}   
t_{_C}(\xi)\ \overset{def.}{=}\ \text{sup}\left\{t>0\ |\ \text{d}(x,\gamma_{_{\xi}}(t))\ =\ t\right\}  
\end{equation}

When $t_{_C}(\xi)<\infty$, the injectivity radius is the least value of $t$ such that $\gamma_{_{\xi}}|_{_{[0,t]}}$ is minimal. The injectivity radius, inj$_{_{\mathcal{M}}}(x)$, is the largest $r$ such that Exp$_{_x}$ is a  diffeomorphism of the open ball of radius $r$ in $T_{_x}\mathcal{M}$ onto its image  

\begin{equation}    
\begin{aligned}
\text{inj}_{_{\mathcal{M}}}(x)\ &\overset{def.}{=}\ \text{d}(x, Cx)\\    
&\ =\ \text{inf}\left\{t_{_C}(\xi)|\xi\ \in\ U_{_x}\mathcal{M}\right\}  
\end{aligned}  
\end{equation}

Together with the systole function introduced in Section \ref{sec:1}, the injectivity radius is a metric invariant. From its definition, it provides a quantitative way of estimating the accuracy with which a tangent space to a given manifold can approximate the manifold itself. The intermediate range in between these two invariants is referred to as the generalised systole function, sys$_{_p}$, which is essentially a pointwise version of $f_{_{\text{sys}}}$  

\begin{definition}   
\begin{equation}    
\text{sys}_{_p}\ \overset{def.}{=}\ \underset{\gamma\in\pi_{_1}(S,p)}{\text{inf}}\ l_{_{\gamma}}  
\end{equation}  
namely the infimum of geodesic loops based at $p$.  

\end{definition}  

For the purpose of our work, the definition of a localised systole function will turn out being crucial in the context of GW-invariant calculations for complete intersections in projective spaces, as we will be explaining more in depth in Section \ref{sec:6}. 

\subsection{Comments on topological recursion} \label{sec:3}

In Case 1, ev$^{^*}(\alpha_{_i})$ is independent w.r.t. $h$, implying topological recursion still works, up to a combinatorial factor in front. However, when $\alpha_{_i}=\alpha_{_i}(h)$, topological recursion needs to be constrained.

\medskip  

\medskip  

\begin{proposition}  Topological recursion already includes nodal invariants, only that they are being summed over and accounted for combinatorially because they all provide a plausible pair of pants subtraction to calculate $V_{_{g,n}}$ recursively. For the case of nodal curves, topological recursion needs to be applied to each node separately.  
\end{proposition}

\begin{proof}   See Appendix \ref{sec:B}, and \cite{Mir}.  
\end{proof}

\begin{proposition} Topological recursion is constrained by primitive insertions.
\end{proposition}

  \begin{proof}    The proof of this statement follows from some observations that can be made by looking at a concrete example. Let us jointly consider Figure \ref{fig:12} and the right-hand-side of Figure \eqref{fig:lamination}. Applying \eqref{eq:McM} to $L_{_1}$, we clearly see that some of the simple closed geodesics associated to pants decompositions, would be crossing the simple cohomology insertions replacing the primitive cohomologies. 
   \begin{figure}[ht!]    
\begin{center}
    \includegraphics[scale=0.5]{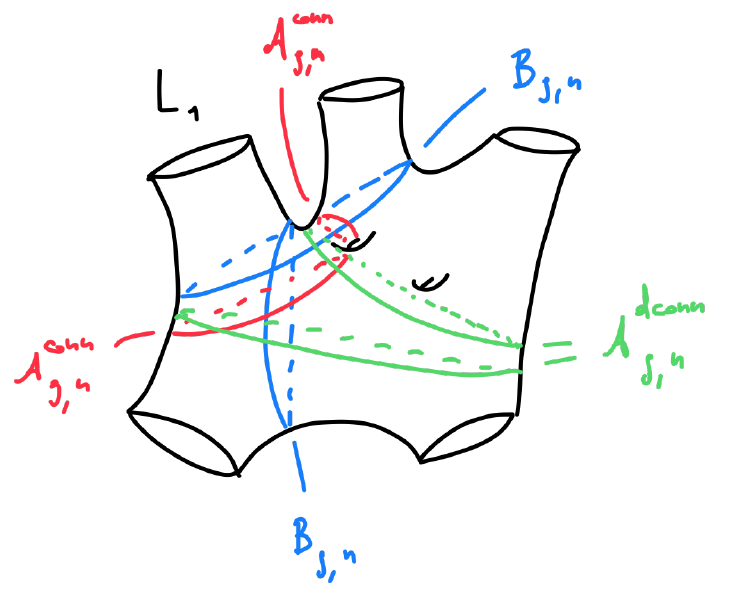}    
    \end{center}   
\caption{\small This picture shows the different possible pants decompositions with respect to the reference closed geodesic $L_{_1}$. Notation used here matches that of Mirzakhani, \cite{Mir}.}
\label{fig:12}  
\end{figure} 
\end{proof}

 In what follows, we will be using the following terminology:\emph{preferred lamination}, $\mathcal{L},$ by which we mean the union of simple closed curves replacing the primitive cohomology insertions. To the best of our knowledge, such terminology has not appeared in the literature before, and we thereby expect it not to lead to any overlap with other settings. To such lamination, $\mathcal{L}$, we can associate a family of admissible pair of pants decompositions, $\{\mathcal{P}\}_{_{\mathcal{L}}}$, that keeps ${\cal L}$ fixed. Of course, this implies limiting the number of possible decompositions that one could use for describing topological recursion in the framework of \cite{Mir}, which thereby leads us to formulating the following statement.

\begin{lemma}  The number of constituents of $\{\mathcal{P}\}_{_{\mathcal{L}}}$ increases with the genus of $\mathcal{M}_{_{g,n}}$ and decreases with the number of primitive insertions.   
\end{lemma}

\begin{proof}   The number of primitive insertions increases the number of fixed simple geodesics in $\mathcal{P}$, thereby reducing the combinatorial rearrangement of the simple geodesics in a given pair of pants decomposition. 
\end{proof}

\begin{corollary} The number of constituents of $\{\mathcal{P}\}_{_{\mathcal{L}}}$ corresponds to the number of different possible sets of generators of MCG$_{_{g,n}}$.
    
\end{corollary} 

\begin{proof}  This can be proved by means of techniques introduced in Section \ref{sec:5}.
\end{proof}

\section{\textbf{Geodesic laminations in groups}}    \label{sec:5}

\begin{definition}A geodesic lamination, $\lambda$, is a set of disjoint union of simple complete geodesics which form a closed subset of $S$, where each geodesic is a \emph{leaf}  
\begin{equation}  
\mathcal{G}\ \subset\ \mathbb{H}^{^2}  \nonumber
\end{equation} 
$\lambda$ is a \emph{minimal} geodesic lamination, $\mathcal{MGL}$, if $\lambda$ contains no proper sub-laminations, namely iff every leaf is dense. A lamination is \emph{filling} if it decomposes $S$ into contractible or 1-punctured surfaces. A filling lamination is said to be \emph{maximal} if it decomposes $S$ into triangles.  

\end{definition}  

\begin{proposition}  Laminations labelling families of pair of pants decompositions, $\{\mathcal{P}\}_{_{\mathcal{L}}}$, can be completed to a filling lamination. The MCG containing the original lamination .
\end{proposition} 
\begin{proof}       Using the stabiliser of MCG$_{_{g,n}}$ and its derivation from monodromy. A more detailed explanation of this follows from the analysis of Section \ref{sec:6}.
\end{proof}  

\begin{proposition} The generators of MCG$_{_{g,n}}$ are dimensionally related to the Thurston spine.
\end{proposition}   
\begin{proof} This statement sheds light on the role played by the stabiliser of $MCG_{_{g,n}}$, and readily follows from the observations featuring in the concluding part of Section \ref{sec:1}. 
\end{proof}

\begin{lemma}  Gromov-Witten invariants of complete intersections in projective space whose symplectic form depends on primitive cohomological insertions are labelled by laminations.   
\end{lemma}

\begin{proof}     The lamination labelling GW-invariants in the setup specified in the statement, can be recast to that of the pants decomposition, $\{\mathcal{P}\}_{_{\mathcal{L}}}$.
\end{proof}

\begin{corollary}   Gromov-Witten invariants are labelled by families/classes of laminations.  

\end{corollary}

\begin{proof} The proof of this statement simply follows from observing that laminations intrinsically feature in calculations of GW-invariants in the form of pair of pants decompositions, with the latter leading to the definition of the Weil-Petersson volume form, \cite{Mir}, \ref{sec:B}.
\end{proof}

\begin{theorem}  The preferred lamination, $\mathcal{L}$, replacing primitive cohomology insertions is not a critical point of the systole function, $f_{_{\text{sys}}}$.  
\end{theorem}

\begin{proof}     Primitive cohomologies lead to unstable homotopy groups, thereby corresponding to points other than the fixed points of $f_{_{\text{sys}}}$. This requires using the systole function at a point, sys$_{_p}$.
\end{proof}   

\begin{corollary}   There is no possible interpolation between different laminations, since the virtual cohomological dimension differs in the two cases.  

\end{corollary}

\begin{proof}  As explained in Section \ref{sec:6}, this follows from the fact that different laminations replacing primitive cohomology insertions correspond to different obstruction bundles, different Ext-groups, and, therefore, different behaviours of the profile of the perturbation near any such point in a given singular variety.
\end{proof}

\subsection{Hyperbolic groups}

\begin{definition} A hyperbolic group is a type of group in Geometric Group Theory characterised by certain properties related to hyperbolic geometry, specifically, if it acts properly-discontinuously and cocompactly by isometries on a proper hyperbolic metric space.  
\end{definition}
The free product of two hyperbolic groups is also a hyperbolic group, and a subgroup of finite index in a hyperbolic group is also hyperbolic. 
Gromov's surface subgroup conjecture predicts that every 1-ended $\delta$-hyperbolic group contains a subgroup isomorphic to the fundamental group of an orientable surface of some genus $g\ge 2$.   
\begin{definition} An important class of hyperbolic spaces is provided by finitely-generated groups whose Cayley graphs are Gromov hyperbolic spaces. Such groups are called word hyperbolic.   
\end{definition}
Geometric group theory exploits the connection between abstract algebraic properties of a group and geometric properties of a space on which such group acts by isomorphisms. 

The topology of a manifold can be controlled by requiring the gluing data to live in various special subgroups of MCG  
\begin{equation}  
K\ <\ \text{MCG}  \nonumber
\end{equation}  
implying the analysis is brought down to determining whether a given element $f\in\text{MCG}$ given as a product of generators of MCG (such as a composition of Dehn twists) lies in a given subgroup, $K$. Identifying such an algorithm is referred to as the \emph{generalised word problem} for $K$ in MCG (a classical problem introduced by [Margulis] and [Nielson].)     
A more refined problem is that of expressing a given element of MCG in terms of generators of a subgroup.    
This is quantitatively encoded in the notion of group distorsion, essentially comparing the word metric on MCG and $K$. 
In the context of an exact sequence of hyperbolic groups  
\begin{equation}  
1\ \longrightarrow\ H\ \longrightarrow\ G\ \longrightarrow\ Q\ \longrightarrow\ 1   
\end{equation}  
algebraic ending laminations arise in different context, such as when determining the existence of Cannon-Thurston maps, \eqref{eq:CT}, or in presence of closed geodesics exiting the ends of a manifold.    

\begin{definition} To every Gromov hyperbolic space, $X$ (Specifically a word-hyperbolic group ), one can associate the space at infinity, $\partial X$. Such boundary is crucial in relating the decomposition of the hyperbolic groups to the topological properties of the boundary.    
\end{definition}

\begin{definition}   
\begin{equation}
\mathcal{L}\ \subseteq\ \partial^{^2}H\ =\ \left(\partial H\times \partial H\right)\backslash \Delta
\end{equation}    
where $\Delta$ denotes the diagonal of $\partial H\times \partial H$, with $\partial H$ equipped with the Gromov topology. 
\end{definition}

\begin{theorem}   [Mitra] Let $A$ be a finite graph of groups where all vertex and edge groups are word-hyperbolic and all edge groups are quasi-isometric embeddings. Let $G=\pi_{_1}(A, v_{_o})$, and suppose $G$ is word-hyperbolic. Then, for any vertex group, $H$, of $A$, the inclusion 
\begin{equation}  
H\ \longrightarrow\ G  \nonumber
\end{equation}  
induces a continuous Cannon-Thurston map  
\begin{equation}  
\hat i:\ \partial H\ \longrightarrow\ \partial G .    
\label{eq:CT}     
\end{equation}  
\end{theorem}  

A key step to establish this is to construct, given an $H$-geodesic segment, $\lambda$, a certain set $B_{_{\lambda}}$ containing $\lambda$ in the Cayley graph $\Gamma(G)$ and then defining a Lipschitz retraction  
\begin{equation}  
\pi_{_{\lambda}}:\ \Gamma(G)\ \longrightarrow\ B_{_{\lambda}}  \nonumber
\end{equation}    
with $B_{_{\lambda}}$ quasiconvex in $G$. $B_{_{\lambda}}$ can be used to show that, if $\lambda$ is away from the identity element in the $H$-metric, it is still away from it in the $G$-metric, thereby implying the existence of a Cannon-Thurston map.

For a hyperbolic group $H$ acting properly on a hyperbolic metric space $X$, for example $X$ could be a Cayley graph of a hyperbolic group $G$ containing $H$, choose a generating set of $H$. If $X$ is a Cayley graph, a generating set of $H$ can be extended to one of $G$, assuming a natural inclusion map  
\begin{equation}  
i:\ \Gamma_{_H}\ \longrightarrow\ \Gamma_{_G}  \nonumber
\end{equation}  

Choosing a basepoint, $*$, the orbit map from the vertex set of $H$ to $X$ which sends $h$ to $h*$ is denoted by  
\begin{equation}  
i:\ h\ \mapsto\ h*  \nonumber
\end{equation}  
and can also be extended to the edges of $\Gamma_{_H}$ by sending them to a geodesic segment in $X$.

\subsection{Generating MCG$_{_{g,n}}$}    

Let    
\begin{equation}  
G\ \simeq\ \pi_{_1}\left(S_{_{2g}}\right)\ \longrightarrow\ \Gamma_{_g}   
\end{equation}   
be an injective map. Consider $\partial_{_{\infty}}(G)$, canonically identified with the circle at infinity of the universal cover    
\begin{equation}  
\tilde S_{_{2g}}\ \simeq\ \mathbb{H}^{^2}  \nonumber
\end{equation}  
of $S_{_{2g}}$. Does there exist a continuous $G$-equivariant map  
\begin{equation}  
\partial_{_{\infty}}(G)\ \longrightarrow\ \mathbb{P}\mathcal{ML}_{_o}(\Sigma)  
\end{equation}

A main question that naturally arises in this setup is the following. 

\textbf{Question:} How does the lamination replacing primitive cohomology insertions relate to $\mathbb{P}\mathcal{ML}$?  

For the time being, such question will be left unanswered in the present work, but we plan to report about it in a followup work, making use of additional tools that were not planned to be used in this instance.   

A classical result by Nielson and Thurston is the classification of mapping class group elements, which can be stated in the following succinct manner:

\begin{theorem}  [Nielson, Thurston] Classification of mapping classes. Given a surface $\mathcal{M}_{_{g,n}}$, let  
\begin{equation}  
h:\ \mathcal{M}_{_{g,n}}\ \longrightarrow\ \mathcal{M}_{_{g,n}}  
\end{equation}   
be a homomorphism. Then, at least one of the following is true:  

\begin{itemize}  

\item $h$ is periodic, i.e. some power of $h$ is equal to the identity.  

\item $h$ is reducible, i.e. $h$ preserves some finite union of disjoint simple closed curves on $\mathcal{M}_{_{g,n}}$.  (Dehn twists are one such example).

\item $h$ is pseudo-Anosov, i.e. no power of $h$ fixes any curves on $\mathcal{M}_{_{g,n}}$.  

\end{itemize}  

\end{theorem}      

\begin{definition} [Thurston] An element $f\in S_{_{g,n}}$ is said to be pseudo-Anosov if there is a representative homeomorphism, $\phi$, a number $\lambda>1$, and a pair of transverse measured foliations, $\mathcal{F}^{^{\Pi}}$, and $\mathcal{F}^{^{\int}}$, such that  
\begin{equation}  
\phi\left(\mathcal{F}^{^{\Pi}}\right)\ =\ \lambda\ \mathcal{F}^{^{\Pi}}\ \ \ \ \ \ \ \ \ \ , \ \ \ \ \ \ \ \ \ \ \phi\left(\mathcal{F}^{^{\int}}\right)\ =\ \lambda^{^{-1}}\ \mathcal{F}^{^{\int}}  \nonumber
\end{equation}
with $\lambda$ denoting the stretch factor (or dilation) of $f$, and $\mathcal{F}^{^{\Pi}}$, $\mathcal{F}^{^{\int}}$ the unstable and stable foliations, respectively. $\phi$ is a pseudo-Anosov homeomorphism.  
\end{definition}

A natural question to ask at this stage is whether there is a systematic way to construct all pseudo-Anosov mapping classes. To better understand this, we make use of Dehn twists mapping classes.

\begin{theorem}  [Dehn]    All mapping classes of $\mathcal{M}_{_{g,n}}$ can be written as products of Dehn twists.  

\end{theorem}

Let $A=\{a_{_1},...,a_{_n}\}$ and $B=\{b_{_1},...,b_{_k}\}$ be multicurves, not necessarily closed, on $\mathcal{M}_{_{g,n}}$, in minimal position, such that $A\cup B$ is filling. Then, any product of positive half-twists about $a_{_j}$ and negative half-twists about $b_{_k}$ is pseudo-Anosov,provided that all $n+k$ Dehn twists feature in the product at least once. For concreteness, let us consider the two cases depicted in Figure \ref{fig:twocases}.  

\begin{figure}[ht!]  
\begin{center}  
\includegraphics[scale=0.4]{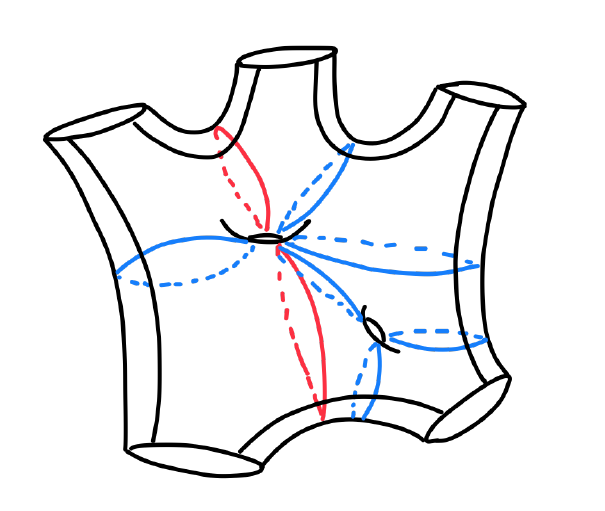}    \ \ \ \ \ \ \ \ \ \ \ \ \  \ \ \ \ \ 
\includegraphics[scale=0.435]{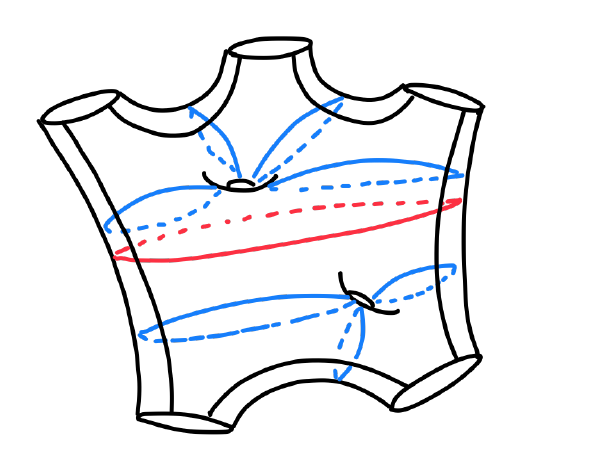}     
\end{center}      
\caption{\small}   
\label{fig:twocases}
\end{figure}

In such case, the multicurves are $A=\{a_{_1},...,a_{_8}\}$ (in red and blue) and $B=\{b_{_1},...,b_{_5}\}$ (in black). $A\cup B$ is maximal in both cases. However, upon applying a Dehn twist on each simple closed curve constituting the filling lamination, 
\begin{equation}  
\phi\ \overset{def.}{=}\ \prod_{_{i=1}}^{^k}D_{_{\mu_{_i}}}^{^{q_{_i}}}  
\end{equation}  
with $\{\mu_{_i}\}$ denoting partitions of the punctures, $1\ \le\ k\ \le\ n$.

A characterisation of a pseudo-Anosov map visualised in polygonal representation is as follows: no curve on the surface is preserved under $\phi^{^n}, \forall n\in\mathbb{N}$.  

As an aside comment, we state the following expectation.

\emph{Conjecture:} The lamination fixing the primitive cohomology splitting of the moduli space must be related to the geodesic axis being preserved under the action of a given pseudo-Anosov map.     

We refrain from providing any rigorous proof in support of such claim, given that it will not be playing any crucial role in the remainder of our work.

\begin{definition}  A projective measured lamination $\lambda\ \in\ \mathcal{PML}$ is filling if $i(\lambda,\gamma)>0$ for every simple closed curve $\gamma$; a filling lamination $\lambda$ is generic if $i(\lambda, \mu)>0, \forall \mu$ on the surface.    
\end{definition}  

\begin{definition} A pseudo-Anosov mapping class $f\in\mathcal{MCG}(S)$ is said to be generic if its stable lamination $\lambda^{^+}$ is generic. The stable lamination of a pseudo-Anosov mapping class if filling. 
\end{definition}

\textbf{Question:} What does the curve graph look like in presence of primitive cohomology insertions? Does it even make sense to construct one? What are the corresponding constraints on it?

\emph{Conjecture:} The translation length diverges between two pseudo-Anosov maps in presence of primitive cohomology insertions in the moduli space.

\begin{proposition}  In presence of primitive cohomology insertions, the lamination associated to a pair of pants decomposition is not a generator of the MCG.
\end{proposition}

\begin{proof}   This follows from the fact that $\mathcal{L}$ is mapped to elements in the stabiliser of MCG$_{_{g,n}}$ and lie in the kernel of the map to the automorphisms of the second homotopy (cf. Section \ref{sec:6}).  
\end{proof}

\begin{theorem}\label{th:2}  Determining MCG$_{_{g,n}}$ for $\mathcal{M}_{_{g,n}}^{^{\Gamma}}$ with primitive cohomology insertions is an example of an unsolvable word problem, featuring a nonrecursive distorsion for subgroups associated to each node of the graph, $\Gamma$.
\end{theorem}
\begin{proof}    \emph{Question:} How to solve the word problem in presence of primitive cohomological insertions? Making use of the finite virtual cohomological dimension. Reason why the GW-calculation might be of use. The problem is unsolvable if considering a subgroup associated only to 1 of the 2 sides of the factorised moduli space. But if we are able to identify the other side as well, then the tensor product of subgroups, a possible lamination defining the generators of MCG could be obtained once having chosen suitable Dehn twist powers. Of course, a suitable gluing should be enforced at the gluing lamination between the two components of the moduli space, most likely a constraint on the intrinsic/extrinsic distance, and distorsions, such as exponential behaviour, etc. 
\end{proof}   

\begin{theorem} [Penner]  Assume $A=\{\alpha_{_1},...,\alpha_{_m}\}$ and $B=\{\beta_{_1},...,\beta_{_n}\}$ be two multicurves that fill a surface $S$. If $f$ is the word made from the product of positive Dehn twists about $\alpha_{_i}$ and negative Dehn twists about $\beta_{_j}$, where all $\alpha_{_i},\beta_{_j}$ are used at least once, then $f$ is a pseudo-Anosov map.  
\end{theorem}    

\begin{definition} The distance on the curve graph $\mathcal{C}\left(\mathcal{M}_{_{g,n}}\right)$, is denoted by d$_{_{\mathcal{C}}}(\alpha,\beta)$. The asymptotic translation length (oftentimes referred to as the stable translation length) of a MCG element $f\in \text{MCG}_{_{g,n}}$ on the curve graph is    

\begin{equation}  
\ell_{_{\mathcal{C}}}(f)\ \overset{def.}{=}\ \underset{n\rightarrow\infty}{\lim}\ \frac{\text{d}_{_{\mathcal{C}}}(\alpha,f^{^n}(\alpha)}{n}  
\label{eq:alphaf}   
\end{equation}  
where $\alpha$ is any vertex of $\mathcal{C}(S)$. \eqref{eq:alphaf} is independent with respect to the choice of $\alpha$.
\end{definition}

\begin{figure}[ht!]    
\begin{center}    
\includegraphics[scale=0.5]{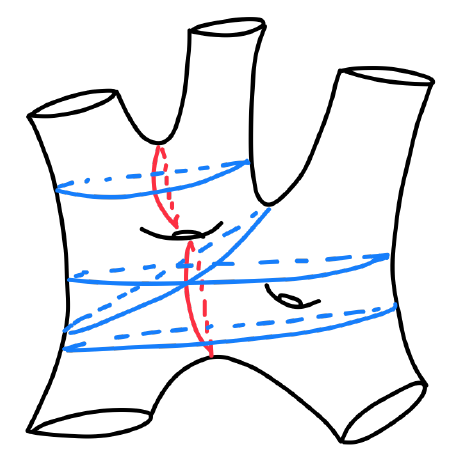}   
\end{center}  
\label{fig:16}   
\caption{\small }
\end{figure}

\begin{enumerate}  

\item Draw all the forbiddden simple closed curves (up to homotopy) that would cross the cohomological insertions.

\item Count how many of them give a pants decomposition, $k$.

\item To each one satisfying 2), corresponds a Dehn twist that can be used as a building block for the generators of MCG$_{_{g,n}}$ of $\mathcal{M}_{_{g,n}}$ without the primitive cohomology insertions.  

\item

\begin{equation}   
h+g_{_1}+g_{_2}-k  \nonumber
\end{equation}  
is the number of simple closed curves that would not be giving a pair of pants decomposition alone.     
\end{enumerate}

\textbf{Question:} What do these simple closed curves correspond to from the point of view of MCG elements (from a GGT point of view)?\color{black}

Because some group elements need to be removed to achieve the new MCG$_{_{g,n}}^{^p}$ with primitive insertions. 

Key point: the primitive cohomological insertions take into account the simple closed curves in MCG$_{_{g,n}}$ that need to be removed to achieve the generators of MCG$_{_{g,n}}^{^p}$, with simple cohomology insertions.

\subsection{Subgroups of MCG$_{_{g,n}}$}

Complete intersections in projective space provide an example of the study of different surface subgroups in MCG$_{_{g,n}}$.  

\begin{equation}   
\text{MCG}^{^p}_{_{g,n}}\bigg\backslash \left(\text{MCG}^{^{(1)}}_{_{g_{_1},n_{_1}}}\ \sqcup\ \text{MCG}^{^{(2)}}_{_{g_{_2},n_{_2}}}\right)\ =\ \bigsqcup_{_{i=1}}^{^{h-1}}\ \text{MCG}^{^{\Gamma_{_i}}}_{_{g,n}}
\end{equation}

\begin{theorem}  

\begin{equation}  
MCG_{_{g,n}}^{^{\Gamma}}\ <\ MCG_{_{g,n}}^{^p}  \label{eq:multigamma}
\end{equation}  
with the left-hand-side featuring constrained pair of pants decompositions.
\end{theorem}

\begin{proof}    Increasing the number of primitive cohomology insertions constraints more and more the number of possible pair of pants decompositions, corresponding to the generators of MCG. Some possibilities must be ruled out if not accounting for the simple closed curves replacing primitive insertions. If different choices of generators are not equivalent,  the corresponding MCG is also different. Most likely, the MCG in question is a nonseparable subgroup of MCG$_{_{g,n}}^{^p}$. 
\end{proof} 

\begin{proposition}  Different primitive cohomology insertions correspond to different subgroups of the MCG; in particular, they correspond to different $f_{_{\text{sys}}}$ on $\mathcal{T}_{_{g,n}}$, that cannot be smoothly interpolated among each other, thereby featuring different critical points, etc.  
\end{proposition}

\begin{proof}     
\end{proof}  

For Case 1, \eqref{eq:multigamma} reduces to 

\begin{equation}  
\text{MCG}_{_{g,n}}\big\backslash \text{MCG}^{^{(1)}}_{_{g_{_1},n_{_1}}}\ =\ \text{MCG}^{^{(2)}}_{_{g_{_2},n_{_2}}}   
\end{equation}

\begin{figure}[ht!]  
\begin{center}  
\includegraphics[scale=0.6]{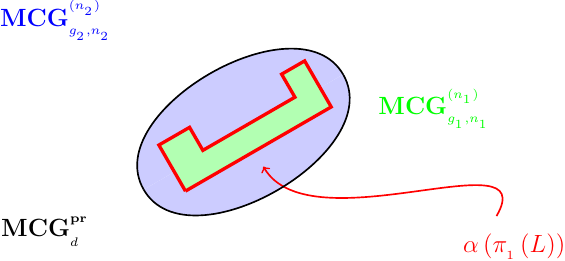}   
 \caption{\small For Case 1, \ref{Word:case1}, MCG splits into two parts and the interface in between the different subgroups is an element of the stabiliser. In the pre-image of the map \eqref{eq:monbeg}, it corresponds to the fundamental group of a simple closed curve, $L$.}
\label{fig:concave}     
\end{center}  
\end{figure}

\begin{figure}[ht!]  
\begin{center}  
\includegraphics[scale=0.5]{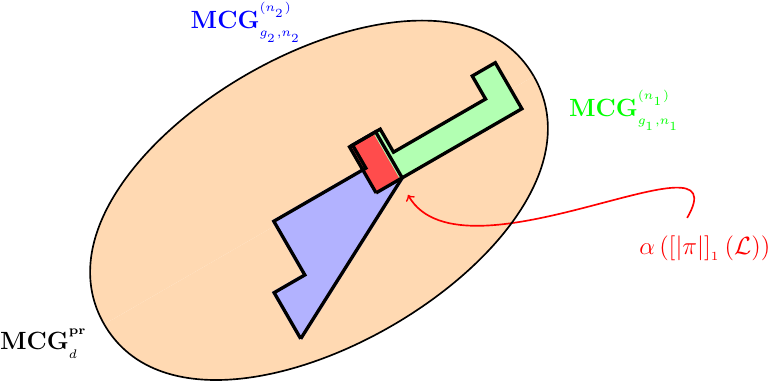}   
 \caption{\small For Case 2, \ref{Word:case2}, MCG can be split into different pairs of subgroups whose intersection is a subgroup of the stabiliser itself, whose pre-image of the map \eqref{eq:monbeg}, it corresponds to a the fundamental germ of a lamination, $\mathcal{L}$. The complement of the union of the two subgroups corresponds to the different other possible choices of the graph dressing $\mathcal{M}_{_{g,n}}$.}
 \label{fig:concave1} 
\end{center}  
\end{figure} 

The notation in Figure \ref{fig:concave} and \ref{fig:concave1} will be made much more rigorous in Section \ref{sec:6}. In particular, with regard to the choice of the point class, which we still haven't introduced in this context. 

\subsection{Intrinsic and extrinsic distances}

\begin{proposition}
\begin{equation}  
MCG_{_{g,n}}^{^{\Gamma}}\ <\ MCG_{_{g,n}}^{^p}  
\end{equation}  
with the left-hand-side featuring constrained pair of pants decompositions, due to the presence of an underlying lamination fixing a multicurve replacing primitive cohomologies.
\end{proposition}

\begin{proof}  It is instructive to consider Figure \ref{fig:concave1}.

\end{proof}

\begin{definition} The distance in the pair of pants decomposition complex, $\mathcal{P}$, 
\begin{equation}    
l_{_{\mathcal{P}}}(f)\ \overset{def.}{=}\ \lim_{_{n\rightarrow\infty}}\frac{\text{d}_{_{\mathcal{P}}}\left(\alpha, f^{^n}(\alpha)\right)}{n}  
\end{equation}   
\end{definition}   

Combining this with the treatment in Section \ref{sec:2}, we are led to the following

\begin{definition}  An intrinsic and extrinsic distance on MCG$_{_{g,n}}^{^{\Gamma}}$ with respect to the ambient group MCG$_{_{g,n}}^{^p}$. 
\end{definition}

\begin{proof}     This follows from the fact that every finitely-generated subgroup of MCG$_{_{g,n}}^{^p}$ is open in the profinite topology.
\end{proof}

\begin{lemma}  The number of constituents of the preferred family of pair of pants decompositions, $\{\mathcal{P}\}_{_{\mathcal{L}}}$, labelled by a lamination, $\mathcal{L}$, replacing primitive cohomology insertions, increases with the genus of $\mathcal{M}_{_{g,n}}$, and decreases with the number of primitive insertions.   
\end{lemma}   
\begin{proof}  By induction.  The number of pairs of pants is 
\begin{equation}  
-\chi\left(\mathcal{M}_{_{g,n}}\right)\ =\ 2g+n-2  \nonumber
\end{equation}     
and they can be combinatorially rearranged by means of A- and S- moves, thereby giving rise to a family of possible decompositions, $\{\mathcal{P}\}$.
If $g^{^{\prime}}>g$, then  
\begin{equation}  
-\chi\left(\mathcal{M}_{_{g^{^{\prime}},n}}\right)\ >\ -\chi\left(\mathcal{M}_{_{g,n}}\right)  
\end{equation}  
therefore the number of constituents in a given decomposition increases with the genus. The number of simple closed curves, instead, reads  
\begin{equation}  
3g+n-3.  \nonumber
\end{equation} 
Upon fixing $2k$ simple closed curves, $\mathcal{L}$, we are left with fewer combinations of pair of pants decompositions, where some of the A and S moves will be forbidden due to the obstruction placed by the $2k$ fixed curves, $\{\mathcal{P}\}_{_{\mathcal{L}}}$.
\end{proof}

\section{\textbf{Unstable homotopy theory}}  \label{sec:6}    
The original motivation of the author to investigate such invariants from an alternative mathematical perspective was driven by the need to better understand the homological mirror symmetry correspondence, where GW-invariants play a crucial role. Among the major advancements in the  field of symplectig geometry in the past decade was the introduction of shifted symplectic structures as a very promising candidate towards generalising notions such as Lagrangian intersections, \cite{PTVV}. We thereby open this concluding Section by briefly overviewing a concrete example where this language arises in the specific context of the calculation of invariants for complete intersections in projective space, \cite{KKS}, explaining how it relates to our findings. 
\subsection{Obstruction bundle} 
Examples of symmetric obstruction theories (SOTs) are the following:  
\begin{itemize}  
\item The critical point of a regular function on a smooth variety. In the present work, the function is question is $f_{_{\text{sys}}}$.   
\item The intersection of Lagrangian submanifolds of a complex symeplectic manifold.  
\end{itemize}    
Both examples are clearly pertinent to Homological Mirror Symmetry applications, and, indeed, we will briefly comment in this regard later on in this Section.

We start from briefly overviewing the geometric setup of \cite{KKS}. (As a side comment, we add that the analysis of \cite{KKS} was mostly driven by a different question from ours, namely that of addressing the question of S-duality arising in Superstring Theory.) Given a smooth projective CY3-variety, $X$. Assuming Pic$(X)$ is generated by an ample divisor $L$. For a fixed integer $k\ \in\ \mathbb{Z}_{_{>0}}$, $H\ \in\ |kL|$, we get a sheaf of $H$-modules over the $H$-hypersurface, whose fiber jumps in dimension upon crossing $H$.   
We need to define the moduli space of semistable sheaves   
\begin{equation}   
\mathcal{M}\left(X, \text{Ch}^{^{\bullet}}\right),    
\label{eq:mch}
\end{equation}  
where $\text{Ch}^{^{\bullet}}$ denotes the \emph{Chow homology}. The calculation of the moduli space \eqref{eq:mch} is usually quite involved. Let $\ell$ be the generator of 
\begin{equation}    
H^{^4}(X;\mathbb{Z})\ \simeq\ \mathbb{Z}.  \nonumber
\end{equation}   
By Poincare' duality, we can define $H^{^2}$.
Given a pair of fixed integers  $i,n\in\mathbb{Z}$, the Chow homology reads as follows  
\begin{equation}  
\text{Ch}^{^{\bullet}}(i,n)\ =\ \left(0,[H], \frac{H^{^2}}{2}-i\ell,\chi(\mathcal{O}_{_H})-H\frac{\text{tod}(X)}{2}-n\right). \nonumber 
\end{equation}   
Semi-stability usually implies that the moduli space only deforms in the linear system. However, in most instances, \eqref{eq:mch} is not smooth. If they were, deformation of the sheaf would give the whole bundle. To circumvent this issue, we need to embed the moduli space into a smoothing 
\begin{equation}   
\mathcal{M}\ \hookrightarrow\ \mathcal{A}_{_{smooth}}  
\end{equation}  
In principle, the smoothing features infinite cohomologies. However, we can identify a map, $\phi$, between the cotangent structures, with the target admitting at most two nontrivial cohomologies  
\begin{equation} 
\mathbb{E}^{^{\times}}\ \xrightarrow{\phi}\  \mathbb{A}.  
\end{equation}  

Subsequently, one can take a point $p\in\mathcal{M}\left(X, \text{Ch}^{^{\bullet}}\right)$ such that

\begin{equation}  
\text{rank}\left(\mathbb{E}^{^{\times}}\right)\bigg|_{_p}\ =\ \text{Vdim}_{_{\mathbb{C}}}\left(\mathcal{M}\left(X, \text{Ch}^{^{\bullet}}\right)\bigg|_{_p}\right)  .   
\label{eq:rankE}  
\end{equation}   
When $X$ is Calabi-Yau, \eqref{eq:rankE} is trivial. The neighbourhood of the selected point in the singular moduli space, $R$, can therefore be plotted in terms of Ext$^{^1}(F,F)$ and Ext$^{^2}(F,F)$, defining the obstruction bundle shown in the figure below, with $F$ denoting a class associated to the point $p$   
\begin{equation}  
[F]\ =\ p\ \in\mathcal{M}\left(X, \text{Ch}^{^{\bullet}}\right)  
\end{equation}  
\begin{figure}[ht!]    
\begin{center}
\includegraphics[scale=0.5]{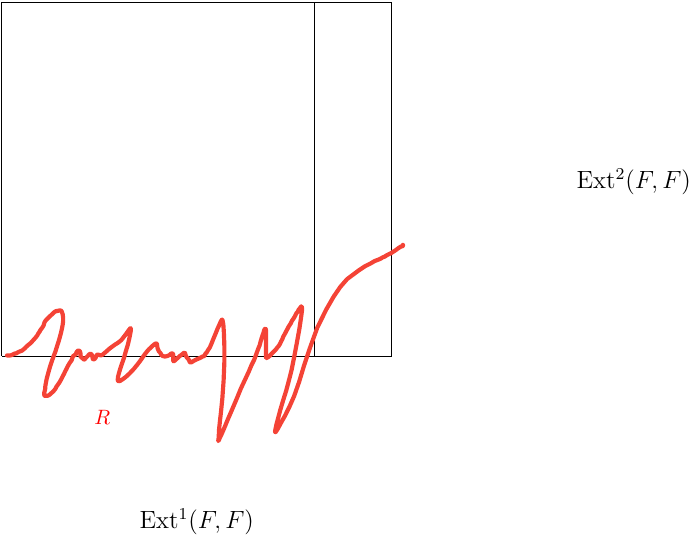}   
\caption{\small Sample sketch of an obstruction bundle for smoothing a singular moduli space. The notation might be slightly deceiving. $R$ actually only denotes the intersection points of the red curve with the axes. Teh red curve corresponds to the deformation of a neighbourhood of the point class of interest, $[F]=p$, under linear perturbation.}
\label{fig:obstrbun} 
\end{center}
\end{figure} 
To the neighbourhood of the deformation, $R$, there is an assigned cycle, $[R]$, which, in turn can be decomposed as follows  
\begin{equation}  
[R]\ =\ \text{number}\ \cdot\ [\text{pt}],    \nonumber
\label{eq:classR}   
\end{equation}  
which, as a class, should be invariant under deformation. Just to define the notation in \eqref{eq:classR}, [pt] corresponds to the point class, whereas the prefactor are nothing but the intersection numbers featuring in Gromov-Witten invariants. Setting some further notation, 
\begin{equation}   
\text{deg} [R] \ =\ [M]^{^{\text{vir}}},  \nonumber
\end{equation}  
we can therefore define the DT-invariants associated to the moduli space in question, which reads as follows  
\begin{equation}  
\text{DT}\left(X, \text{Ch}^{^{\bullet}}\right)\ =\ \int_{_{[M]^{^{\text{vir}}}}}1\ =\ 1\cap[M]^{^{\text{vir}}},    \nonumber  
\end{equation}  
The crucial point to observe here is the role played by the choice of the point class in \eqref{eq:rankE}. From Section \ref{sec:1}, we know that the following chain of equalities holds independently of the choice of the point class  
\begin{equation}   
\begin{aligned}   
\text{dim}\mathcal{P}_{_{g,n}}\ =\ \text{dim}\Phi\left(\mathcal{T}_{_{g,n}}\right)\ &=\ \text{VCohdim}\ \text{Stab}_{_{\text{MCG}}}\ \left(\ell_{_{X_{_d}}}^{^{\text{hyp}}}\right) \\
&=\ \text{VCohdim}\ \text{MCG}
\end{aligned}
\label{eq:strict}
\end{equation}
for any  
\begin{equation}  
p\ \in\ \mathcal{M}\left(X, \text{Ch}^{^{\bullet}}\right)  \nonumber
\end{equation}
However, in the context of complete intersections in projective space, \eqref{eq:strict} is too restrictive, and we need to turn to the generalised systole function, instead, which carries much more richness from the perspective of geometric topology.

\subsection{Monodromies not preserving the symplectic form}  

We now start gathering most of the results achieved in the previous Sections. Let
\begin{equation}  
\mathcal{U}\ \subset\ \prod_{_{i=1}}^{^r}\ \mathbb{P}\left(H^{^0}\left(\mathbb{P}^{^{m+r}},\mathcal{O}\left(d_{_i}\right)\right)\right) \nonumber
\end{equation}  
be an open subset parameterising smooth $m$-dimensional complete intersections of degrees $\left(d_{_1},...,d_{_r}\right)$ in $\mathbb{P}^{^{m+r}}$.  
Fixing a point $u\ \in\ \mathcal{U}$, and denoting by $X$ the corresponding smooth complete intersection in $\mathbb{P}^{^{m+r}}$. 
A class $\gamma\ \in\ H^{^{\bullet}}(X)$ is simple if it lies in the image of $H^{^{\bullet}}\left(\mathbb{P}^{^{m+r}}\right)$and primitive if $\gamma\ \in\ H^{^m}(X)_{_{\text{primt}}}$. The fundamental group $\pi_{_1}(\mathcal{U}, u)$ acts on the primitive cohomologies, $H^{^m}(X)_{_{\text{primt}}}$. The monodromy group therefore reads as follows  
\begin{equation}  
G\ \overset{def.}{=}\ \overline{\Psi(\pi_{_1}(\mathcal{U},u))}  \nonumber
\end{equation}   
where 
\begin{equation}  
\Psi:\ \pi_{_1}(\mathcal{U},u)\ \longrightarrow\ \text{Aut}\left(H^{^{\bullet}}(X)_{_{\text{primt}}}\otimes\mathbb{C}\right). 
\end{equation}  
This is true for Case 1, \ref{Word:case1}, namely for the case in which the monodromy completely determines the primitive cohomologies, \cite{ABPZ}.
From previous work by\cite{LMP}, we have the following Lemma, essentially providing the definition of Zariski closure of a group  

\begin{lemma}[LMP]  
Let $\Gamma$ be a finitely generated group, $G$ a omplex algebraic group, and $\Delta_{_o}$  a finitely-generated subgroup of $\Gamma$. If $\Delta_{_o}$ is strongly distinguished be a representation $\rho_{_o}\in\text{Hom}(\Gamma, G)$, then there exists a representation, $\Psi$, such that   
\begin{equation}  
\Psi(\Gamma)\ \cap\ \overline{\Psi\left(\Delta_{_o}\right)}\ =\ \Psi\left(\Delta_{_o}\right).  
\end{equation}  

\end{lemma}

For the purpose of our work, $\Gamma\equiv$ MCG$_{_{g,n}}$, which is finitely-generated by the Dehn twists on simple closed curves constituting a lamination of $\mathcal{M}_{_{g,n}}$.
\begin{proposition}   The action of primitive cohomologies on the symplectic form descend from the lack of Zariski closure of the subgroups associated to the two different factors in $\mathcal{M}_{_{g,n}}$.
\end{proposition}  
\begin{proof}  In presence of primitive cohomologies affecting the symplectic form, the group, $\Gamma$, is no longer separable, since there are finitely-generated subgroups that are not Zariski-closed. Their number corresponds to the number of nodes in the graph dressing the moduli space.
\end{proof}

\subsection{Zariski closure and subgroup separability}  

To find the Zariski closure of a set, one can consider the set of all points that satisfy the polynomial defining the ideal. E.g., given an ideal generated by the polynomials, the closure will include all points in the affine space that are common roots to these polynomials. 

In general, the Zariski topology is characterised by its coarse nature, where closed sets are larger with respect to the standard topology, making the closure of an ideal encompass a significant portion of the space.

The algebraic monodromy group is defined by the Zariski closure of the image of $\pi_{_1}(U,u)$ in the algebraic group   
\begin{equation}  
\text{Aut}\left(H^{^{\bullet}}(X)_{_{\text{prim}}}\otimes\mathbb{C}\right) \nonumber 
\end{equation}  

The explicit monodromy group $G$ depends on the parity of dim$(X)$, as already addressed by \cite{ABPZ}. Specifically, for odd dimensions, it is $SP(V)$, whereas in the even-dimensional case, it is $O(V)$. 

\begin{theorem}\label{th:1}   MCG is non-separable if the simple cohomologies depend on the nodes (namely if the symplectic form depends on the primitive cohomology insertion).

\end{theorem}  

\begin{proof}  This follows from the fact that every finitely-generated subgroup of MCG$_{_{g,n}}^{^p}$ is open in the profinite topology.

\end{proof}

\subsection{A geometric group theory perspective} 

As previously observed by \cite{RW}, for complete intersections in projective space, the image of the group homomorphism  
\begin{equation}  
\alpha:\ \text{Mon}_{_d}\ \longrightarrow\ \text{MCG}_{_d}  
\label{eq:alpha}
\end{equation}  
features an upper and a lower bound. In terms of the notation adopted by \cite{RW}, these subgroups read  
$\text{Stab}_{_{\text{MCG}_{_d}}}\left(\ell_{_{X_{_d}}}^{^{\text{hyp}}}\right)$, and $\text{Im}\left(\Phi:\ \Theta_{_7}\ \longrightarrow\ K_{_d}\right)$, respectively, with the latter surjecting onto $\text{Aut}\left(\pi_{_3}\left(X_{_d}\right),\lambda,\mu\right)$.   

\begin{itemize}  

\item  The upper bound corresponds to the minimal distance between any two points in MCG$_{_d}$ can be achieved by making use of the generators of one of its subgroups, image of the monodromy map. 

\item The lower bound corresponds to the maximal distance between any two points in MCG$_{_d}$ achieved by making use of the generators of one of its subgroups, when the word problem can be solved.  

\end{itemize}  

Given the context of our analysis, we are therefore brought to pose the following

\textbf{Question:} How do primitive cohomologies enter in this upper versus lower bound? Can we get any further interesting geometric topology insight?  

Primitive cohomologies feature in the data providing lower bound to the map \eqref{eq:alpha}. $\mu$ is Mon$_{_d}$-independent, corresponding to the assumption of independence between the symplectic form and the primitive cohomologies. $\mu$ corresponds to the primitive cohomologies, since it corresponds to the data arising from the embedding $X_{_d}\subset\mathbb{CP}^{^4}$, where $\lambda$ is treated as the symplectic form on the space $\pi_{_3}\left(X_{_d}\right)$. Essentially, primitive cohomologies are repsonsible for increasing the distance between two points in MCG$_{_d}$. The path can be shortened once the generators of $\text{Aut}\left(\pi_{_3}\left(X_{_d}\right),\lambda,\mu\right)$ can be completely rewritten in terms of those of $\text{Aut}\left(\pi_{_3}\left(X_{_d}\right),\lambda\right)$.

\begin{theorem}   

The following statements are equivalent:

\begin{itemize}  

\item MCG is non-separable if the simple cohomologies depend on the nodes (namely if the symplectic form depends on the primitive cohomology insertion).

\item  Identifying MCG$_{_{g,n}}$ for $\mathcal{M}_{_{g,n}}^{^{\Gamma}}$ with primitive cohomology insertions is an example of an unsolvable word problem, featuring a nonrecursive distorsion for subgroups associated to each node of the graph, $\Gamma$, dressing the moduli space.

\end{itemize}    

\end{theorem}

\begin{proof}  This follows from Theorem \ref{th:1} and Theorem \ref{th:2}.
\end{proof}

\begin{definition}
    Let  
    \begin{equation}  
    h:\ X\ \rightarrow\ X  \nonumber
    \end{equation}  
    be a continuous map on a compact metric space, $X$. The topological entropy, $h_{_{\text{top}}}(f)$, is defined as follows  
    \begin{equation}  
    h(f)\ =\ \underset{\epsilon\rightarrow0}{\lim}\ \underset{n\rightarrow\infty}{\lim\text{sup}}\ \frac{\log(N(\epsilon,n))}{n}  ,   
    \end{equation}   
where $N(\epsilon,n)$ denotes the minimal number of open sets of the form $B_{_{\epsilon}}^{^n}(x_{_i})$, $x_{_i}\in X$, covering $X$.

\end{definition}

\begin{theorem}  [Gromov, Yomdin] 
M a compact K$\ddot{\text{a}}$hler manifold, with a surjective holomorphic map  
\begin{equation}  
f:\ M\ \longrightarrow\ M.  \nonumber
\end{equation}  
Then,  the topological entropy reads  
\begin{equation} 
h_{_{\text{TOP}}}(f)\ =\ \log\ \rho\left(f^{^{\bullet}}\right),  \end{equation}  
where $\rho\left(h^{^*}\right)$ denotes  the spectral radius of   
\begin{equation}  
f^{^{\bullet}}:\ H^{^{\bullet}}\left(M;\mathbb{C}\right)\ \longrightarrow\ H^{^{\bullet}}\left(M;\mathbb{C}\right)  \end{equation}
\end{theorem}  

\textbf{Conjecture} \cite{KT} $X$ smooth proper over $\mathbb{C}$, the topological entropy reads

\begin{equation}  
h_{_o}\left(\Phi\right)\ =\ \log\ \rho\left(HH_{_{\bullet}}\left(\Phi\right)\right),  
\end{equation}    
where   

\begin{equation}  
HH_{_{\bullet}}\left(\Phi\right):\ HH_{_{\bullet}}\left(X\right)\ \longrightarrow\ HH_{_{\bullet}}\left(X\right)  
\end{equation}
and $\rho\left(HH_{_{\bullet}}\left(\Phi\right)\right)$ its spectral radius.

\begin{proposition} [Fan]\cite{F} For every even integer $d\ge 4$, let $X$ be a CY hypersurface in $\mathbb{CP}^{^{n+1}}$ of degree $(d+2)$ and    

\begin{equation}  
\Phi\ =\ T_{_{{\mathcal{O}}_{_X}}}\ \circ\ \left(-\otimes\ \mathcal{O}(-1)\right)  
\end{equation}  

Then, its topological entropy is bounded from below as follows   

\begin{equation}  
h_{_o}(\Phi)\ >\ 0\ =\ \log\ \rho\left(\Phi_{_{H^{^{\bullet}}}}\right)  
\end{equation}  
thereby providing a counterexample of [K.,T.]'s Conjecture.

\cite{DHKK} showed that  the categorical entropy on the Fukaya category

\begin{equation}  
h_{_o}\left(D^{^b}\text{Fuk}(X)\right)\ =\ \log\lambda  
\end{equation}
 is always positive, since the stretch factor $\lambda>1$. Hence, from HMS arguments, we should expect to be able to find examples where \cite{KT} conjecture does not hold on $D^{^b}\text{Coh}(X)$.

\end{proposition}  

The discrepancy in between the two matches that between complex geometry and symplectic geometry. K3-surfaces provide a further example of a variety not satisfying \cite{KT}'s Conjecture. We are therefore led to the following

\begin{theorem}  The jump in the topological entropy of the pseudo-Anosov map generating MCG$_{_{g,n}}^{^{\Gamma}}$ is due to a discrepancy between intrinsic and extrinsic distance among elements of MCG$_{_{g,n}}^{^{\Gamma}}$ with respect to the ambient group MCG$_{_{g,n}}^{^{\text{p}}}$ (cf. Figure \ref{fig:concave1}).
\end{theorem} 

\begin{proof} The jump in the topological entropy is related to reinstalling primitive cohomologies. The particular choice of graph, $\Gamma_{_i}$ dressing $\mathcal{M}_{_{g,n}}$ corresponds to removing the primitive cohomology insertions. As previously stated, the choice of $\Gamma$ corresponds to removing the primitive cohomological insertions, ensuring a degeneration formula can be found.   Different choices of graphs lead to different possible degeneracies, all converging towards achieving degeneracy in the topological entropy  
\begin{equation}  
h_{_{\text{TOP}}}\ =\ 0  \nonumber
\end{equation}   
Requiring the simple cohomologies to depend on the graph, specifically on the vertices, selects one particular descendant theory, making the theory more constrained.  

The impact of the two possible choices on MCG$_{_{g,n}}$ is the following:  

\begin{itemize}   
\item     
In case the simple cohomologies are independent w.r.t. the primitive insertions, both sides in the following relation
\begin{equation}   
\text{MCG}_{_{g,n}}\ <\ \text{MCG}_{_{g,n}}^{^p}  \nonumber
\end{equation}   
are separable. If the simple cohomologies are independent w.r.t. the vertices of $\Gamma$, all possible choices are equivalent, thereby leading to no new advancement from the point of view of the word problem. Thus, from the geometric group theory point of view, this setup leads to the same topological recursion relations.
\item   On the other hand, in absence of such dependence, both sides in
\begin{equation}   
\text{MCG}_{_{g,n}}^{^{\Gamma_{_i}}}\ <\ \text{MCG}_{_{g,n}}^{^p}  \nonumber
\end{equation}   
are non-separable. This case is much more interesting from the geometric group theory analysis, and it is reasonable to expect this to be playing a significant role in determining new topological recursion relations (cf. Section \ref{sec:2}).
\end{itemize}   
\end{proof}
\begin{theorem}  
Any complete intersection in projective space whose MCG$_{_{g,n}}^{^p}$ is separable, features a pseudo-Anosov map whose topological entropy is bounded from below by the logarithm of its spectral radius.
\end{theorem}   

\begin{proof}  Making use of MCG separability.   
Consider $X_{_d}\subset\mathbb{CP}^{^4}$, with complex dimension $d=3$. The pseudo-Anosov map  
\begin{equation}  
h^{^{\bullet}}:\ H^{^{\bullet}}\left(X_{_d};\mathbb{C}\right)\ \longrightarrow\ H^{^{\bullet}}\left(X_{_d};\mathbb{C}\right)  \nonumber
\end{equation} 

At the homotopical level, it reads   

\begin{equation}  
h^{^{\bullet}}:\ \pi_{_3}\left(X_{_d}\right)\ \longrightarrow\ \pi_{_3}\left(X_{_d}\right)  \nonumber 
\end{equation}
but $\pi_{_3}$ also inherits a skew-symmetric form $\lambda$ from the ambient space. 
MCG$^{^p}_{_{g,n}}$ is separable if Mon$_{_d}$ can fully determine it. Specifically, the upper bound to the map 
\begin{equation}  
\alpha:\ \text{Mon}_{_d}\ \longrightarrow\ \text{Stab}\left(\ell_{_{X_{_d}}}^{^{\text{hyp}}}\right)\ \le\ \text{MCG}_{_d}^{^p} 
\end{equation}  
where $\text{Stab}\left(\ell_{_{X_{_d}}}^{^{\text{hyp}}}\right)$ is well-defined for the case in which MCG$_{_{g,n}}^{^p}$ is separable.
\end{proof}  

\begin{theorem}  
If MCG$_{_{g,n}}^{^p}$ is non-separable, $h^{^{\bullet}}$ is not pseudo-Anosov.
\end{theorem}  

\begin{proof}  Making use of MCG non-separability. 
Non-separability implies 
\begin{equation}   
\text{Stab}\left(\ell_{_{X_{_d}}}^{^{\text{hyp}}}\right)\ \longrightarrow\ \text{Aut}\left(\pi_{_3}\left(X_{_d}\right),\lambda,\mu\right)  
\end{equation}  
is not a surjection, the reason being that the skew-symmetric form $\lambda$ is now dependent on the quadratic refinement, $\mu$, induced by the immersion $X_{_d}\ \subset\ \mathbb{CP}^{^4}$.  
Specifically, 
\begin{equation}  
\alpha:\ \text{Mon}_{_d}\ \longrightarrow\  \text{MCG}_{_d}^{^p}  
\end{equation} 
no longer features $\text{Stab}\left(\ell_{_{X_{_d}}}^{^{\text{hyp}}}\right)$ as an upper bound, and therefore  
\begin{equation}  
\text{Stab}\left(\ell_{_{X_{_d}}}^{^{\text{hyp}}}\right)\ \not\le\ \text{MCG}_{_d}^{^p}  
\end{equation}
\end{proof}  

\subsection{Residual finiteness}  

\begin{definition} Residual finiteness is a topological property, stating that, given $G=\pi_{_1}(X)$ and $C$ a compact subspace of the universal cover $\tilde X$ of $X$, then there exists a finite cover $Y$ of $X$ such that $\tilde X\rightarrow\ Y$ is injective on $C$. 
\end{definition}    
Residual finiteness is a property of groups in which every nontrivial element can be distinguished from the identity by a finite quotient. A group is thereby residually-finite if for any non-identity element, there exists a finite group such that the image of the element is not the identity in that group.  
\begin{theorem}  For complete intersections in projective space, residual finiteness of MCG only holds if the primitive cohomologies are independent w.r.t. the symplectic form.
\end{theorem}  
\begin{proof}  Two possible proofs can be provided. 
The first readily follows from the relation between residual finiteness and separability (together with Theorem \ref{th:1}).  

Alternatively, one might argue by contradiction. Suppose the reverse is true. Assuming the primitive cohomologies depend on the symplectic form. Then, we know that Stab$\left(\ell_{_{X_{_d}}}^{^{\text{hyp}}}\right)$ does not inject in Aut$\left(\pi_{_3}\left(X_{_d}\right),\lambda,\mu\right)$, and 
\begin{equation}  
\text{Stab}\left(\ell_{_{X_{_d}}}^{^{\text{hyp}}}\right)\ \not\le\ \text{MCG}_{_{g,n}}^{^p}  \nonumber
\end{equation}
due to the underlying unstable homotopy theory. Im$(\alpha)$ is not residually-finite, thereby neither is MCG$^{^p}_{_{g,n}}$.
Residual finiteness is associated to the pseudo-Anosov feature of multi-curves, of which Dehn twists can be taken to define generators of MCG. If even one simple closed curve that is part of the aforementioned multicurve does not behave as pseudo-Anosov, MCG of the corresponding variety is not residually finite. The latter corresponds to having some closed curves as fixed points of automorphisms of the moduli space, and therefore corresponds to the case in which primitive cohomologies depend on the symplectic form.

\end{proof}

\subsection{The fundamental germ}     
The fundamental germ is a generalisation of $\pi_{_1}$, first defined in \cite{G} for laminations arising through group actions, and subsequently extended to any lamination having a dense leaf admitting a smooth structure. It is therefore the fundamental group for laminations.

The idea underlying its construction relies on the ansatz that some features of manifold theory can be extended to laminations by replacing a manifold object by a family of manifold objects existing on the leaves of a lamination, $\mathcal{L}$, and respecting the transverse topology, $T$. However, such correspondence when attempting to construct objects such as tensors, de Rham cohomology groups, falls short from being exactly one-to-one, due to the discreteness of the setup involving the lamination. 

For example, considering the case of an exceptionally well-behaved lamination, one could define an inverse limit  of manifolds by covering maps
\begin{equation}  
\hat{M}\ = \ \underset{\longleftarrow}{\lim}\ M_{_{\alpha}}    
\label{eq:mhat}   
\end{equation}  
Such a system induces a direct limit of de Rham cohomology groups, and the limit extends to cohomology groups $H^{^{\bullet}}(\hat{M};\mathbb{R})$ with dense image. However, from \eqref{eq:mhat}, the construction of other groups, such as $\pi_{_1}$, $H^{^{\bullet}}(\hat{M};\mathbb{Z})$, $H_{_{\bullet}}(\hat{M};\mathbb{R})$ does not follow through so straightfarwardly.     

For this reason, for certain classes of laminations, \cite{G} proposed a new construction of the fundamental group of a lamination, referred to as the fundamental germ, and conventionally denoted as follows  

\begin{equation}  
[|\pi|]_{_1}(\mathcal{L},x)  \nonumber
\end{equation}  

The intuition guiding such construction is that of considering laminations as an irrational manifold.

Consider a pointed manifold $(M,x)$. Its fundamental group can be viewed as the identifications to be made in the universal cover $(\tilde M, \tilde x)$ to achieve $(M,x)$ as a quotient. Points in an orbit of the fundamental group correspond to equivalence classes.   
Under this perspective, the generalisation of the fundamental group consists in replacing $\pi_{_1}$-orbits with mutual approximation of points by means of an auxiliary transverse space, $T$. In such case, the universal cover $(\tilde M,\tilde x)$ no longer produces a quotient manifold, but, rather a leaf, $(L,x)$, of a lamination, $\mathcal{L}$. The germ of the transversal space $T$ about $x$ may therefore be interpreted as the failed attempt of $(L,x)$ to form an identification topology at $x$. The fundamental germ $[|\pi|]_{_1}(\mathcal{L},x)$ is therefore a device that algebraically records the dynamica of $(L,x)$ as it approaches $x$ through the topology of $T$.

\begin{figure}[ht!]  
\begin{center}  
\includegraphics[scale=1]{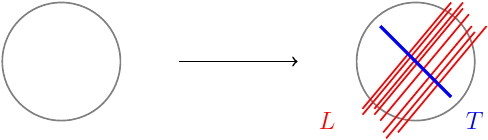}    
\caption{\small Reproduction of a Figure in \cite{G}, essentially showing how to turn an orbit in the fundamental group to a sequence of intersections in the transverse space. Identify orbits of $\pi_{_1}$ (on the left-hand-side), approximate them through a transversal, $T$ (on the RHS), and keep track of the sequences in $\pi_{_1}$ whose $x$-translates converge transversally to $x$.} 
\label{fig:orbit}    
\end{center}  
\end{figure}  

Making use of Figure \ref{fig:orbit}, we can state the following 
\begin{definition}  
The fundamental germ, $[|\pi|]_{_1}(\mathcal{L},x)$ is the set of tail equivalence classes of sequences of the form  
\begin{equation}  
\{g_{_{\alpha}}h_{_{\alpha}}^{^{-1}}\}  \nonumber
\end{equation}   
such that   
\begin{equation}  
g_{_{\alpha}}x, h_{_{\alpha}}x\ \longrightarrow\ x\nonumber
\end{equation}   
in $T$.   
\end{definition}  
By component-wise multiplication, the fundamental germ comes equipped with a grupoid structure, and   
\begin{equation}   
\pi_{_1}(L,x)\ <\ [|\pi|]_{_1}(\mathcal{L},x)   
\end{equation}
Of the two cases introduced in Section \ref{sec:2}, we recall their respective MCG structure, previously mentioned in Section \ref{sec:5}

\underline{Case 1}   

\begin{equation}  
\text{MCG}_{_{g,n}}\big\backslash \text{MCG}^{^{(1)}}_{_{g_{_1},n_{_1}}}\ =\ \text{MCG}^{^{(2)}}_{_{g_{_2},n_{_2}}}   \nonumber
\end{equation}

\underline{Case 2}   

\begin{equation}  
\text{MCG}_{_{g,n}}\big\backslash\left(\ \text{MCG}^{^{(1)}}_{_{g_{_1},n_{_1}}}\ \cup\ \text{MCG}^{^{(2)}}_{_{g_{_2},n_{_2}}}\ \right)\ = \ \bigsqcup_{_{i=1}}^{^{h-1}}\text{MCG}^{^{\Gamma_{_i}}}_{_{g,n}} \nonumber   
\end{equation}
From Section \ref{sec:5}, we have that
\begin{equation}  
\text{Ker}\left( \text{Stab}_{_{\text{MCG}_{_{g,n}}}}\left(\ell_{_{\mathcal{M}}}^{^{[F]}}\right)\ \longrightarrow\ \text{Aut}\left(\pi_{_2}\left([F]\right),\lambda\right)\right)\ \equiv\ \alpha\left([|\pi|]_{_1}(\mathcal{L}), [F]\right)     
\label{eq:Ker}   
\end{equation}   
with    
\begin{equation}  
\alpha\left([|\pi|]_{_1}(\mathcal{L}), [F]\right) \ <\   \text{Stab}\left(\ell_{_{\mathcal{M}}}^{^{[F]}}\right)
\end{equation}
\eqref{eq:Ker} implies $\text{Stab}_{_{\text{MCG}_{_{g,n}}}}\left(\ell_{_{\mathcal{M}}}^{^{[F]}}\right)\ \longrightarrow\ \text{Aut}\left(\pi_{_2}\left([F]\right),\lambda\right)$ is surjective. Hence, $\alpha\left([|\pi|]_{_1}(\mathcal{L}),[F]\right)$ correspond to the elements  that need to be removed from the stabiliser. This follows from the fact that they do not inject in $\text{Aut}\left(\pi_{_3}\left(X_{_d}\right),\lambda\right)$, meaning they are missing in the automorphisms
\begin{equation}  
H^{^{\bullet}}\left(X_{_d}\right)\ \longrightarrow\ H^{^{\bullet}}\left(X_{_d}\right)    
\end{equation}
thereby leading to the identification \eqref{eq:Ker}. 

$K_{_d}$ being the lower bound on the subgroups of MCG$_{_{g,n}}^{^{\text{pr}}}$ that can be achieved by means of the map $\alpha$, we are led to the following statement.

\begin{theorem}   For complete intersections in projective space,  
\begin{equation}   
\alpha\left([|\pi|]_{_1}(\mathcal{L}), [F]\right)\ \le\ \text{Im}(\alpha)\ \le\ \text{Stab}_{_{\text{MCG}_{_{g,n}}}}\left(\ell_{_{{\mathcal{M}}}}^{^{[F]}}\right)
\end{equation}     
with the equality being achieved when $\mathcal{L}$ corresponds to a pants decomposition of $\mathcal{M}_{_{g,n}}$. This essentially corresponds to a saturation limit. 
 
\end{theorem}

\begin{proof}    
Deformation retraction provides a consistency check in support of this statement. In the degenerate case, the the Thurston spine would read as follows  
\begin{equation}  
\alpha\left([|\pi|]_{_1}(\mathcal{L}), [F]\right)\left(\mathcal{T}_{_{g,n}}\right)\ \equiv\ \mathcal{P}_{_{g,n}}\Bigg|_{_{[F]}}     
\end{equation}  
with dimension    
\begin{equation}  
\begin{aligned}
\text{dim}\ \mathcal{P}_{_{g,n}}\Bigg|_{_{[F]}}   \ &=\ \text{VCohdim}\left( 
\alpha\left([|\pi|]_{_1}(\mathcal{L}), [F]\right)\right)\ \equiv\   \text{VCohdim}\left( K_{_d}\right)\\  
&=\ \text{VCohdim}\left( 
\text{Ker}\left( \text{Stab}_{_{\text{MCG}_{_{g,n}}}}\left(\ell_{_{\mathcal{M}}}^{^{[F]}}\right)\ \longrightarrow\ \text{Aut}\left(\pi_{_2}\left([F]\right),\lambda\right)\right)\right)
\end{aligned}    
\label{eq:newth}   
\end{equation}    

For a complete intersection (not embedded in an ambient complex projective space), \eqref{eq:newth} is trivially vanishing, since  
\begin{equation}   
\text{Ker}\left( \text{Stab}_{_{\text{MCG}_{_{g,n}}}}\left(\ell_{_{\mathcal{M}}}^{^{[F]}}\right)\ \longrightarrow\ \text{Aut}\left(\pi_{_2}\left(\mathcal{M}_{_{g,n}}\right),\lambda\right)\right) \ \equiv\{\empty\}  
\end{equation}    
which is exactly what we would expect.

On the other hand, away from the limiting case,    
\begin{equation}   
\text{VCohdim}\left( 
\alpha\left([|\pi|]_{_1}(\mathcal{L}),[F]\right)\right)\ <\ \text{VCohdim}\left(\text{Stab}_{_{\text{MCG}_{_{g,n}}}}\left(\ell_{_{\mathcal{M}}}^{^{[F]}}\right)\right)\ \equiv\ \text{VCohdim}\left(\text{MCG}_{_{g,n}}^{^{\text{pr}}}\right)\Bigg|_{_{[F]}}
\end{equation}   
implying    
\begin{equation}   
\text{VCohdim}\left( 
\alpha\left([|\pi|]_{_1}(\mathcal{L}),[F]\right)\right)\ <\ \text{dim}\ \mathcal{P}_{_{g,n}}\Bigg|_{_{[F]}}   
\end{equation} 

\end{proof}

\textbf{Question}: What is the meaning of the area-like overlap in between the different MCG subgroups in Case 2 (as shown in Figure \ref{fig:concave1})?   

We can interpret is as follows: the nontrivial area at the boundary of two corresponds to having 
\begin{equation}  
\text{Stab}_{_{\text{MCG}_{_{g,n}}}}\left(\ell_{_{\mathcal{M}}}^{^{[F]}}\right)\ \longrightarrow\ \text{Aut}\left(\pi_{_2}\left(\mathcal{M}_{_{g,n}}\right),\lambda\right)  
\end{equation}  
not injective, due to the fact that primitive cohomological insertions depend on the symplectic form.      

In general,   \cite{RW} and \cite{ABPZ} imply considering the representation

\begin{equation}  
\Psi:\ \pi_{_1}(\mathcal U, u)\ \longrightarrow\ \text{Aut}\left(\pi_{_3}\left(X_{_d}\right),\lambda\right)   
\label{eq:longarrow1}
\end{equation}

In Case 1, \cite{ABPZ},  

\begin{equation}   
H^{^{\bullet}}\left(X_{_d}\right)\ \equiv\ H^{^{\bullet}}\left(X_{_d}\right)_{_{\text{prim}}} 
\end{equation}  
implies \eqref{eq:longarrow1} reduces to   

\begin{equation}  
\Psi:\ \pi_{_1}(\mathcal U, u)\ \longrightarrow\ \text{Aut}\left(H^{^{\bullet}}\left(X_{_d}\right)\otimes\mathbb{C}\right)   
\label{eq:longarrow2}
\end{equation} 
Hence, if we were to consider a simple closed curve, $L$, the relation between the fundamental groups would read  

\begin{equation}  
\pi_{_1}(L, u)\ <\ \pi_{_1}(\mathcal{U},u)  
\label{eq:grouincl}   
\end{equation}  

However, replacing the left-hand-side of \eqref{eq:grouincl} with the fundamental germ\footnote{Recall that $\mathcal{L}$ here denotes the lamination replacing the primitive cohomology insertions.}  

\begin{equation}  
[|\pi|]_{_1}(\mathcal{L}, u)  
\end{equation}  
\eqref{eq:longarrow2} reads as follows  
\begin{equation}  
\Psi:\ [|\pi|]_{_1}(\mathcal{L})\ \longrightarrow\ \text{Aut}\left(H^{^{\bullet}}\left(X_{_d}\right)\otimes\mathbb{C}\right)   
\label{eq:longarrow11}
\end{equation}  
with the cohomologies $H^{^{\bullet}}\left(X_{_d}\right)$ now being not necessarily primitive, thereby corresponding to our Case 2 (Section \ref{sec:3}).   
We conclude with a few remarks and observations. The first of which is a fundamental Theorem in the context of the fundamental group.
\begin{theorem}\label{th:VK}[Van Kampen] A fundamental result in Algebraic Topology that describes how to compute the fundamental group of a space from the fundamental group of its subspaces.   
\end{theorem}  
Our analysis brings to light a case where the Van Kampen Theorem is not applicable in its original formulation, and we therefore believe it would be relevant to identify a possible generalisation of its original formulation.   

Given the correspondence between the injectivity radius and the critical exponent, 
\begin{equation}   
\delta(\Gamma)+1\ =\ \text{VCohdim}\left[\alpha\left(\left[|\pi|\right]_{_1}(\mathcal L)\right)
\right]\ =\ \text{dim}\ \mathcal{P}_{_{g,n}}  
\label{eq:critexpeq}
\end{equation}
the first equality in \eqref{eq:critexpeq} actually corresponds to the saturation limit 
\begin{equation}   
\underset{\mathcal{T}(S,p)}{\text{sup}}\ f_{_{\text{sys}}}\ =\ 2\ \underset{\mathcal{T}(S)}{\text{sup}} \ \text{inj}   
\end{equation}
In terms of the generalised systole function, jointly with our analysis, we can write the following generalisation of \eqref{eq:critexpeq}
\begin{equation}   
\delta\left(\text{Stab}_{_{\text{MCG}_{_d}}}\left(\ell_{_{\mathcal{M}}}^{^{[F]}}\right)\right)+1\ \ge\ \text{VCohdim}\left[\alpha\left(\left[|\pi|\right]_{_1}(\mathcal L, [F])\right)
\right]      
\label{eq:critexpeq1}
\end{equation}   
since   
\begin{equation}
    \alpha\left(\left[|\pi|\right]_{_1}(\mathcal L, [F])\right)\ \le\ \text{Stab}_{_{\text{MCG}_{_{g,n}}}}\left(\ell_{_{\mathcal{M}}}^{^{[F]}}\right)
\end{equation}
where equality is correctly recovered in the case of complete intersections in absence of a projective ambient space.

\section*{\textbf{Acknowledgements}}

The author wishes to thank Xiaolong Hans Han for discussions and explanations on hyperbolic geometry and geometric topology. A special acknowledgement also goes to Yi Huang for hospitality while visiting the Geometry and Topology Group at YMSC in Tsinghua, where part of this research was conducted in June 2025. She also acknowledges Ingrid Irmer for sharing Thurston's original unpublished work on spines, \cite{Th}, as well as Qile Chen, Si Li, and Zhengyu Zong for insightful discussions with regard to algebro geometric aspects of this work. Research conducted by the author is funded by SIMIS.

\appendix

\section{\textbf{Virtual cohomological dimension}}

The \emph{virtual cohomological dimension} of a group, $G$, is the largest integer, $n$, for which there exists a $G$-module, $M$, with

\begin{equation}  
H^{^n}(G,M)\ \neq\ 0 \nonumber
\end{equation}     

If   Cohdim$(G)\neq 2$, there exists a cell complex of dimension equal to Cohdim$(G)$. In presence of torsion, the cohomological dimension could diverge. However, one could define an alternative meaningful invariant when $G$ is torsion-free, namely the \emph{virtual} cohomological dimension, VCohdim$(G)$, defined as the cohomological dimension of a torsion-free subgroup, $\hat G$, of finite index in $G$. A Theorem by Serre states that this number is independent w.r.t. $\hat G$. 

One class of groups for which VCohdim is known is the class of virtual duality groups. A group $G$ is referred to as a duality group of dimension $d$ if there exists a $ZG$-module, $I$, and a class, $e$, in $H_{_d}(G;I)$, such that, if $M$ is any other $ZG$-module, the cap product with $e$ induces an isomorphism  
\begin{equation}  
H^{^k}(G;M)\ \xrightarrow{\sim}\ H_{_{d-k}}(G;M\otimes I)  \nonumber
\end{equation}  
with Cohdim$(G)=d$. Arithmetic groups are examples of virtual duality groups.

With ${\cal T}_{_{g,n}}$ being homeomorphic to a $6g-6+2n$-dimensional Euclidean space, and MCG$_{_{g,n}}$ virtually torsion-free, implies the existence of the following isomorphism, \cite{H}, 
\begin{equation}  
H_{_k}\left(\text{MCG}_{_{g,n}};\mathbb{Q}\right)\ \simeq\ H_{_k}\left({\cal M}_{_{g,n}};\mathbb{Q}\right)  \ \  \ \ \ \ \forall\  k.\nonumber
\end{equation}   

The only value of $k$ for which these groups are explicitly known are $k=1,2$.

\section{\textbf{Topological recursion}}\label{sec:B}

The volume of the moduli space    
\begin{equation}  
\mathcal{M}_{_{g,n}}\ \overset{def.}{=}\ \mathcal{T}_{_{g,n}}/\text{MCG}. \nonumber 
\end{equation}
is a polynomial with positive coefficients in surface boundary lengths of total degree $6g-6+2n$.  

\underline{Theorem}\cite{Mir}  [The WP volume polynomials]  
The volume polynomials are determined recursively from the volume polynomials of smaller total degree. The volume of the moduli space $V_{_{g,n}}\left(L_{_1},...,L_{_n}\right)$ of hyperbolic surfaces of genus $g$ and $n$ boundary components, with respective length  

\begin{equation}  
L\ =\ \left(L_{_1},...,L_{_n}\right) \nonumber 
\end{equation}   
is a polynomial   

\begin{equation}  
V_{_{g,n}}(L)\ =\ \underset{|\alpha|\le 3g-3+n}{\sum_{_{\alpha}}}C_{_{\alpha}}L^{^{2\alpha}}. \nonumber  
\end{equation}
where the coefficients read as follows  

\begin{equation}  
C_{_{\alpha}}\ =\ \frac{2^{^{\delta_{_{1g}}\delta_{_{1n}}}}}{2^{^{|\alpha|}}\alpha!(3g-3+n-|\alpha|)!}\int_{_{\overline{\mathcal{M}}_{_{g,n}}}}\psi_{_1}^{^{\alpha_{_1}}}...\psi_{_n}^{^{\alpha_{_n}}}\omega^{^{3g-3+n+|\alpha|}}. 
\end{equation}
These are the GW-invariants.

\begin{equation}     
\frac{\partial}{\partial L^{^{(1)}}}L^{^{(1)}}V_{_{g,n}}^{^l}\ =\ {\cal A}_{_{g,n}}^{^{\text{conn}}}+{\cal A}_{_{g,n}}^{^{\text{dconn}}}+{\cal B}_{_{g,n}}     
\label{eq:McM}   
\end{equation}


\newpage
\begin{thebibliography}{CEG87} 
\bibitem[ABPZ]{ABPZ} H. Argüz, P. Bousseau, R. Pandharipande, D. Zvonkine, Gromov-Witten Theory of Complete Intersections via Nodal Invariants,	arXiv:2109.13323 [math.AG]

\bibitem[Br]{Br}  Nathan Broaddus, Benson Farb, Andrew Putman, Irreducible Sp-representations and subgroup distortion in the
 mapping class group, https://math.uchicago.edu/~farb/papers/distortion.pdf

\bibitem[DHKK]{DHKK} G. Dimitrov, F. Haiden, L. Katzarkov and M. Kontsevich, Dynamical systems and categories,
 Contemporary Mathematics, 621 (2014), 133–170.

\bibitem[FM]{FM} B. Farb, L. Mosher, Convex cocompact subgroups of
 mapping class groups, https://math.uchicago.edu/~farb/papers/schottky.pdf
 
 \bibitem[G]{G}T. M. Gendron, The Geometric Theory of the Fundamental Germ, 	arXiv:math/0506275 [math.DG]

\bibitem[F]{F} Y.-W. Fan, Entropy of an autoequivalence on Calabi-Yau manifolds, arXiv:1704.06957 [math.AG]

 \bibitem[MG]{MG} M. Gromov, "Hyperbolic groups" S.M. Gersten (ed.) , Essays in Group Theory , MSRI Publ. , 8 , Springer (1987) pp. 75–263  

 
  \bibitem[GTV]{GTV} M. Gibson, L.-S. Tseng, and S. Vidussi, Symplectic structures with non-isomorphic primitive cohomology on open 4-manifolds, Transactions of the American Mathematical Society
 Volume 375, Number 12, December 2022, Pages 8399–8422

\bibitem[H]{H} Harer, J.L., The virtual cohomological dimension of the mapping class group of an orientable surface. Invent Math 84, 157–176 (1986). https://doi.org/10.1007/BF01388737   

\bibitem[KT]{KT} K. Kikuta, A. Takahashi, On the categorical entropy and the topological entropy, 	arXiv:1602.03463 [math.AG]

\bibitem[KKS]{KKS} L. Katzarkhov, M. Kontsevich, A. Sheshmani, \emph{to appear}.

\bibitem[Irm]{Irm} Irmer, I., The Morse-Smale property of the Thurston spine


\bibitem[LMP]{LMP} L. Louder, D. B. McReynolds, P. Patel, Zariski Closures and Subgroup Separability, 	arXiv:1510.04144 [math.GR].   


\bibitem[MM]{MM} H. Masur and Y. Minsky, Geometry of the complex of curves. I. Hyperbolicity, Invent.
 Math., 138(1):103-149, 1999.

 


\bibitem[Mir]{Mir} M. Mirzakhani, Simple geodesics and Weil-Petersson volumes of moduli spaces
 of bordered Riemann surfaces. Invent. Math., 167 (2007), no. 1, 179–222.


 \bibitem[MT]{MT} Curtis T. McMullen and Clifford H. Taubes, 4-manifolds with inequivalent symplectic forms
 and 3-manifolds with inequivalent fibrations, Math. Res. Lett. 6 (1999), no. 5-6, 681–696,
 DOI 10.4310/MRL.1999.v6.n6.a8. MR1739225 

\bibitem[N]{Niel} J. Nielsen, Surface transformation classes of algebraically finite type, Danske Vid. Selsk. 
Math.-Phys. Medd. 68 (1944). 


\bibitem[PTVV]{PTVV} T. Pantev, B. Toen, M. Vaquie, G. Vezzosi, Shifted Symplectic Structures, 	arXiv:1111.3209 [math.AG]


\bibitem[RW]{RW} Randal-Williams, O., Monodromy and mapping class groups of 3-dimensional hypersurfaces. Math. Ann. 391, 1965–2003 (2025). https://doi.org/10.1007/s00208-024-02951-4

\bibitem[Th]{Th} Thurston, W.P.: On the geometry and dynamics of di eomorphisms of surfaces.
 Bulletin of the american mathematical society 19(2), 417431 (1988)


  \bibitem[Th1]{Th1} W. P. Thurston, The geometry and topology of three-manifolds, Princeton Math. 
Dept., 1979. 


  \bibitem[Th2]{Th2} W.P. Thurston, Some simple examples of symplectic manifolds, Proc. Amer. Math.
 Soc. 55 (1976), 467–468.
   \bibitem[Th3]{Th3}  W.P. Thurston, A norm for the homology of 3-manifolds. Mem. Amer. Math. Soc. 59 (1986),
 99–130.  

 \bibitem[Th]{Th}  W.P. Thurston, A spine for Teichm$\ddot{\text{u}}$ller space, \emph{unpublished}.

 \bibitem[TY]{TY} Li-Sheng Tseng and Shing-Tung Yau, Cohomology and Hodge theory on symplectic manifolds:
 II, J.DifferentialGeom.91 (2012), no. 3, 417–443. MR2981844





\end{thebibliography}
\end{document}